\documentclass[11pt,a4paper]{article}
\usepackage{amsmath,amsfonts,amssymb,mathrsfs}
\usepackage{graphicx} 
\usepackage{subfigure}
\usepackage{pstricks}
\usepackage{pst-node}
\usepackage{xcolor}
\usepackage{cite}
\usepackage[T1]{fontenc}
\usepackage{cite}
\usepackage{mathscinet}
\usepackage[english]{babel}
\usepackage{soul}

\newrgbcolor{blu}{.1 0 .7}
\newrgbcolor{lgrey}{.9 .9 .9}
\newrgbcolor{grey}{.7 .7 .7}
\newrgbcolor{dgrey}{.3 .3 .3}
\newrgbcolor{lred}{1 .8 .8}
\newrgbcolor{mlred}{1 .7 .7}
\newrgbcolor{bblue}{.6 .6 1}
\newrgbcolor{lblue}{.8 .8 1}
\newrgbcolor{mblue}{.2 .25 1}
\newrgbcolor{dblue}{.3 .32 1}
\newrgbcolor{ddblue}{.4 .4 1}
\newrgbcolor{lgreen}{.8 1 .8}
\newrgbcolor{mgreen}{ .7 1 .7}
\newrgbcolor{ggreen}{ .1 .75 .35}
\newtheorem{theorem}{\bf Theorem}[section]
\newtheorem{lemma}[theorem]{\bf Lemma}

\newtheorem{definition}[theorem]{\bf Definition}
\newtheorem{remark}[theorem]{\bf Remark}
\newtheorem{proposition}[theorem]{\bf Proposition}

\numberwithin{equation}{section}

\newcommand{\R}{\mathbb{R}}

\newcommand{\N}{\mathbb{N}}

\def \rnn {{\mathbb {R}}^{2n+1}}
\def \I {{\mathbb{I}}}
\def \OO {{\mathbb{O}}}

\def \L {\mathscr{L}}

\def \A {\mathscr{A}}

\def \GG {{\mathbb{G}}}

\def \lL {{\mathbb{K}}}

\def \div {{\text{\rm div}}}
\def \tr {{\text{\rm Tr}}}
\def \loc {{\text{\rm loc}}}
\def \p {\partial}

\def \diag {{\text{\rm diag}}}
\def \det {{\text{\rm det}}}

\def \H {\mathcal{Q}}

\newenvironment{sergiorev}{\color{blue}}{\color{black}}
\newcommand{\bsr}{\begin{sergiorev}}
\newcommand{\esr}{\end{sergiorev}}

\newenvironment{frarev}{\color{red}}{\color{black}}
\newcommand{\bfra}{\begin{frarev}}
\newcommand{\efra}{\end{frarev}}

\newcommand{\denoterow}[1]{\rlap{\hspace{1em}$\leftarrow{}$ $j^{th}$}}

\def \D {{\Delta}}

\newenvironment{proof}{\noindent {\sc Proof.}}{\hfill $\square$}

\def \o {{\omega}}
\def \a {{\alpha}}
\def \b {{\beta}}

\def \g {{\gamma}}
\def \d {{\delta}}
\def \e {{\varepsilon}}

\def \epsilon {{\varepsilon}}

\def \k {{\kappa}}
\def \l {{\lambda}}

\def \s {{\sigma}}
\def \t {{\tau}}
\def \m {{\mu}}

\def \x {{\xi}}

\def \z {{\zeta}}
\def \w {{\omega}}
\def \phi {{\varphi}}
\def \G {{\Gamma}}
\def \O {{\Omega}}
\def \ker {{{\rm Ker}}}

\def \AS {{ \mathscr{A}_{(x_{0}, t_{0})} }}

\definecolor{shadecolor}{rgb}{.9, .9, 1}
\newrgbcolor{blu}{.1 0 .7}
\newrgbcolor{lgrey}{.88 .88 .88}
\newrgbcolor{llgrey}{.94 .94 .94}
\newrgbcolor{grey}{.75 .75 .75}
\newrgbcolor{dgrey}{.3 .3 .3}
\newrgbcolor{lred}{1 .8 .8}
\newrgbcolor{mlred}{1 .7 .7}
\newrgbcolor{bblue}{.6 .6 1}
\newrgbcolor{lblue}{.8 .8 1}
\newrgbcolor{mblue}{.2 .25 1}
\newrgbcolor{dblue}{.3 .32 1}
\newrgbcolor{ddblue}{.4 .4 1}
\newrgbcolor{lgreen}{.8 1 .8}
\newrgbcolor{mgreen}{ .7 1 .7}
\newrgbcolor{ggreen}{ .1 .75 .35}

\usepackage{mathtools}
\DeclarePairedDelimiter{\abs}{\lvert}{\rvert}
\DeclarePairedDelimiter{\bnorm}{\parallel}{\parallel}
\DeclarePairedDelimiter{\babs}{\lvert}{\rvert}

\title{A survey on the classical theory for Kolmogorov equation}
\author{{\sc{Francesca Anceschi}
\thanks{Dipartimento di Scienze Fisiche, Informatiche e Matematiche, Universit\`{a} di Modena e Reggio Emilia, Via
Campi 213/b, 41125 Modena (Italy). E-mail: francesca.anceschi@unimore.it} \qquad \sc{Sergio Polidoro}
\thanks{Dipartimento di Scienze Fisiche, Informatiche e Matematiche, Universit\`{a} di Modena e Reggio Emilia, Via
Campi 213/b, 41125 Modena (Italy). E-mail: sergio.polidoro@unimore.it}
}}
\date{}

\begin{document}
\maketitle

\begin{abstract}
    We present a survey on the regularity theory for classic solutions to subelliptic degenerate Kolmogorov equations. 
    In the last part of this note we present a detailed proof of a Harnack inequality and a strong maximum principle.
\end{abstract}

\bigskip
\normalsize
\tableofcontents

\section{Introduction}
\label{sec1}
Kolmogorov equations appear in the theory of stochastic processes as linear second order parabolic equations with 
non-negative characteristic form. Throughout this paper we are mainly concerned with degenerate Kolmogorov equations.
In its simplest form, if $\left(W_t \right)_{t \ge 0}$ denotes a real Brownian motion, the density $p=p(t,v,y,v_{0},y_{0})$ of the stochastic process  $(V_{t}, Y_{t})_{t \ge 0}$
\begin{equation}
   \label{processo}
    \begin{cases}
    V_t = v_{0} + \s W_{t} \\
    Y_t = y_{0} + \int_{0}^{t} V_{s} \, ds
    \end{cases}
\end{equation}
is a solution to a strongly degenerate Kolmogorov equation, that is
\begin{equation}
	\label{sdk}
        \tfrac12 \s^{2} \p_{vv} p + v \p_y p = \p_t p,  \qquad t \ge 0, \qquad
    (v,y) \in \R^2.
\end{equation}
In 1934 Kolmogorov provided us with the explicit expression of the density $p=p(t,v,y,v_{0},y_{0})$ of the above equation (see \cite{K1}) 
\begin{equation}
   \label{sf}
   p(t,v,y, v_{0}, y_{0}) = \tfrac{\sqrt{3}}{2 \pi t^2} 
    \exp \left( - \tfrac{(v-v_{0})^2}{t} - 3 \tfrac{(v-v_{0})(y - y_{0} - t v_{0})}{t^2} - 3 \tfrac{(y - y_{0} - t y_{0})^2}{t^3} \right) \qquad t > 0,
\end{equation}
and pointed out that it is a smooth function despite the strong degeneracy of the equation \eqref{sdk}. As it is suggested by the smoothness of the density $p$, the operator $\L$ associated to equation \eqref{sdk}
\begin{equation} \label{KolmoR3}
	\L := \tfrac12 \s^{2} \p_{vv} + v \p_y - \p_t, 
\end{equation}
is hypoelliptic, in the sense of the following definition, that we state for a general second order differential operator $\L$ acting on an open subset $\O$ of $\R^N$.

\medskip

\noindent 
{\sc Hypoellipticity.} \textit{The operator $\L$ is hypoelliptic if, for every distributional solution $u \in L^1_{\loc}(\O)$  to the equation $\L u = f$, we have that}
    \begin{equation} \label{hypo}
        f \in C^\infty(\O) \quad \Rightarrow \quad u \in C^\infty(\O).
    \end{equation}

\medskip

H\"ormander considered the operator $\L$ defined in \eqref{KolmoR3} as a prototype for the family of hypoelliptic operators studied in his seminal work \cite{H}. Specifically, the operators considered by H\"ormander are of the form 
\begin{equation}
	\label{ol}
	\L = \sum_{k=1}^{m} X^{2}_{k} + Y,
\end{equation}
where $m$ is a natural number and $X_{k}$ are smooth vector fields of the form
\begin{equation}
	\label{vector}
	X_{k} = \sum \limits_{j=1}^{N+1} b_{j,k} (z) \, \p_{z_{j}},
	\qquad Y = \sum \limits_{j=1}^{N+1} b_{j,m+1} (z) \, \p_{z_{j}} \quad k = 1, 	\ldots, m,
\end{equation}
with $b_{j,k} \in C^{\infty}(\O)$ for every $j=1, \ldots, N+1$, $k=1, \ldots, m+1$ and $\O$ is any open subset of $\R^{N+1}$.
The main result presented in \cite{H} is a sufficient condition to the hypoellipticity of $\L$. Its statement requires some notation. Given two 
vector fields $Z_{1}, Z_{2}$, the commutator of $Z_{1}$ and $Z_{2}$ is the vector field:
\begin{equation*}
	[Z_{1},Z_{2}] = Z_{1} \, Z_{2} -  Z_{2} \, Z_{1}.
\end{equation*}
Moreover, we recall that ${\rm Lie} (X_1, \ldots, X_m, Y)$ is the Lie algebra generated by the vector fields 
$X_1, \ldots, X_m, Y$ and their commutators. 

\medskip

\noindent 
{\sc H\"ormander's Rank Condition.} \textit{Suppose that
    \begin{equation}
        \label{rc1}
        {\rm rank} \, {\rm Lie} (X_1, \ldots, X_m, Y)(z) = N+1 \qquad \text{for every} \quad z \in \O.
    \end{equation}
    Then the operator $\L$ defined in \eqref{ol} is hypoelliptic in $\O$,}

\medskip
Let us consider again the operator $\L$ defined in \eqref{KolmoR3} with $\sigma = \sqrt{2}$ to simplify the notation. 
$\L$ can be written in the form \eqref{ol} if we choose 
\begin{equation*}
	X =  \p_{v}  \sim (0, 1, 0)^{T},
	\qquad Y = v \p_y - \p_t \sim (-1, 0, v)^{T},
\end{equation*}
and the H\"ormander's rank condition is satisfied, as
\begin{equation*}
	[X,Y] = XY - Y X = \p_{y} \sim (0, 0, 1)^{T}.
\end{equation*}

\medskip

As the regularity properties of H\"ormander's operators $\L$ are related to a Lie algebra, it became clear that the natural framework for 
the regularity theory of H\"ormander's operators is the non-euclidean setting of Lie groups, as Folland and Stein pointed out in \cite{FS}. 
Later on, Rothschild and Stein developed a general regularity theory for H\"ormander's operators in \cite{RS}. We refer to the more recent 
monograph by Bonfiglioli, Lanconelli and Uguzzoni \cite{BLU} for a comprehensive treatment of the recent achievements of the theory.
We also recall the book \cite{Bramanti} by Bramanti. As far as we are concerned with the operator $\L$, we show that it is invariant with respect to the non-commutative traslation given by the 
following composition law 
\begin{equation*}
    (t,v,y) \circ (t_{0}, v_{0}, y_{0}) = (  t_{0} + t, v_{0} + v, y_{0} + y - t v_{0}), \qquad (t,v,y), (t_{0}, v_{0}, y_{0}) \in \R^{3}.
\end{equation*}
Indeed, if $w(t,v,y) = u ( t_{0} + t, v_{0} + v, y_{0} + y - t v_{0})$ and 
$g(t,v,y) = f( t_{0} + t, v_{0} + v, y_{0} + y - t v_{0})$, then 
\begin{equation*}
	\L u = f \quad \iff \quad \L w = g \quad \text{for  every} \quad (t_{0}, v_{0}, y_{0}) \in \R^{3}.
\end{equation*}
As we will see in the sequel, in several applications the couple $(v,y)$ denotes the velocity and the position of a particle. For this reason 
the above operation is also known as \emph{Galilean} change of variable.

Another remarkable property of the operator $\L$ is its dilation invariance. More precisely, the operator $\L$ is invariant with respect to the 
following family of dilations
\begin{equation*}
    \d_{r} (t,v,y) := (r^2 t, r v, r^3 y), \qquad r >0,
\end{equation*} 
with the following meaning: if we define $ w (t,v,y) = u (r^2 t, r v, r^3 y)$ and $g(t,v,y)=f(r^2 t, r v, r^3 y)$ we have that
\begin{equation*}
    \L u = f \quad \iff \quad \L w = r^2  g \quad \text{for  every} \quad r > 0.
\end{equation*}
As we will see in the sequel, this underlying invariance property plays a fundamental role in the study of the operator $\L$, even though it 
does not hold true for every Kolmogorov operator (see Section \ref{sec2}), as it happens in the family of uniformly parabolic operators. 
Indeed, we usually consider parabolic dilations $\d_{r}(x,t)=(rx, r^{2} t)$ also when considering the Ornstein-Uhlenbeck operator 
$\L = \Delta - \langle x, \nabla \rangle - \partial_t$.

\medskip

We conclude this introduction discussing about some applications of the Kolmogorov equation. First of all, the process \eqref{processo} is the solution to the Langevin equation
\begin{equation*}
    \begin{cases}
    d V_t = d W_{t} \\
    d Y_t =  V_{t} \, dt,
    \end{cases}
\end{equation*}
hence Kolmogorov equations are related to every stochastic process satisfying Langevin equation. In particular, several 
mathematical models involving linear and non linear Kolmogorov type equations have also appeared in finance \cite{AD}, 
\cite{B}, \cite{BP} and \cite{DHW}. Indeed, equations of the form \eqref{sdk} appear in various models for pricing of 
path-dependent financial instruments (cf., for instance, \cite{BPVE} \cite{PA}), where, for example the equation
\begin{equation}
    \label{bpv}
     \p_t P + \tfrac12 \s^{2} S^2 \p^2_S P + (\log S ) \p_{A} P + r (S \p_{S} P - P) = 0, \qquad S > 0, 
    \, A, t \in \R
\end{equation}
arises in the Black and Scholes option pricing problem
\begin{equation*}
    \begin{cases}
    d S_t = \m S_{t} dt + \s S_{t} d W_{t} \\
    d A_t =  S_{t} \, dt,
    \end{cases}
\end{equation*}
where $\s$ is the volatility of the stock price $S$, $r$ is the interest rate of a risckless bond and $P= P(S, A, t)$ is the price of the Asian option depending on the price of the stock $S$, the geometric average $A$ of the
past price and the time to maturity $t$.  

Moreover, we recall that the Kolmogorov equation is the prototype for a family of evolution equations arising in kinetic theory of gases which take the following general form 
\begin{equation}
    \label{kt}
    Y u = \mathcal{J} (u).
\end{equation}
In this case, we have that $u=u(v,y,t)$ is the density 
of particles with velocity $v=(v_{1}, \ldots, v_{n})$ and position $y=(y_{1}, \ldots, y_{n})$ at time $t$. 
Moreover, 
\begin{equation*}
    Y u := \sum \limits_{j=1}^n v_j \p_{y_{j}} u + \p_t u
\end{equation*}
is the so called total derivative with respect to time in the phase space 
$\rnn$. $\mathcal{J} (u)$ is the collision operator, which can be either 
linear or non-linear. For instance, in the usual Fokker-Planck 
equation (cf. \cite{DV}, \cite{R}) we have a linear collision operator  of the form
\begin{equation*}
    \mathcal{J} (u) = \sum \limits_{i,j=1}^n a_{ij} \, \p_{v_{i}, v_{j}}^2 u + \sum 
    \limits_{i=1}^n a_i \, \p_{v_{i}} u + a u
\end{equation*}
where $a_{ij}$, $a_i$ and $a$ are functions of $(x,t)$; $\mathcal{J} (u)$ can
also occur in divergence form
\begin{equation*}
     \mathcal{J} (u) = \sum \limits_{i,j=1}^n \p_{v_i} (a_{ij} \, \p_{v_{j}} u + 
     b_i u ) + \sum \limits_{i=1}^n a_{i} \p_{v_{i}} u + a u.
\end{equation*}
We also mention the following non-linear collision operator of the 
Fokker-Planck-Landau type
\begin{equation*}
    \mathcal{J} (u) = \sum \limits_{i,j=1}^n \p_{v_i} \big(a_{ij} (z,u) \p_{v_{j}} 
    u + b_i(z,u) \big), 
\end{equation*}
where the coefficients $a_{ij}$ and $b_i$ depend both on $z \in \rnn$ and the unknown functions $u$ through some integral expression. 
Moreover, this last operator is studied as a simplified version of the Boltzmann collision operator (cf. \cite{C}, \cite{L}). For the description 
of wide classes of stochastic processes and kinetic models leading to equations of the previous type, we refer to the classical 
monographies \cite{C}, \cite{CC} and \cite{DM}. For further applications we refer to the article \cite{spigler} by Akhmetov, Lavrentiev and Spigler, to the work \cite{tersenov} by Tersenov, and to the references therein.

\medskip
 
The regularity theory for weak solutions to the Kolmogorov equation with measurable coefficients in divergence form
has been developed during the last decade and is still evolving.  
As the aim of this survey is to give an overview of the established theory, we simply recall some of the main results on this subject. Operators with VMO coefficients $a_{ij}$ have been studied in \cite{BramantiCeruttiManfredini} by Bramanti, Cerutti and Manfredini, \cite{ManfrediniPolidoro} by Manfredini and Polidoro, and in
\cite{PolidoroRagusa1}, \cite{PolidoroRagusa2} by Polidoro and Ragusa. 
The theory of Kolmogorov operators with measurable coefficients $a_{ij}$ is developed in the following papers:
\begin{itemize}
	\item \textbf{Moser iteration:} 
		a first contribution is given by Polidoro and Pascucci (see \cite{PP}) for dilation invariant
		Kolmogorov operators with measurable coefficients; later on, Cinti, Pascucci and Polidoro extend 
		this result to the non-dilation invariant case (see \cite{CPP}). Finally, the non-dilation invariant case 
		with lower order coefficients with positive divergence is proved by the authors and Ragusa (see \cite{APR}).
	\item \textbf{Poincaré inequality and H\"older regularity:}
		a weak Poincaré inequality is proved by Wang and Zang in \cite{WZ3} for the dilation invariant case and in 
		\cite{WZ4} for the non-dilation invariant one.  Related results have been recently proved by Armstrong and 
		Mourrat for the kinetic Kolmogorov-Fokker-Planck equation in \cite{AM}.
	\item \textbf{Harnack inequality:} 
	        Golse, Imbert, Mouhot and Vasseur prove the H\"older continuity and a Harnack inequality for weak solutions  
	        to the kinetic Kolmogorov-Fokker-Planck equation (see \cite{GIMV}). Based on their results, 
	        The authors and Eleuteri prove a geometric statement for the Harnack inequality (see \cite{AEP}).
\end{itemize}
As far as we are concerned with regularity theory for weak solutions to the Kolmogorov equation with measurable 
coefficients in non-divergence form, the only result available is due to Abedin and Tralli, who prove a Harnack inequality  
for this type of operators with additional Cordes-Landis assumption on the coefficients $a_{ij}$ (see \cite{AT}).
We finally recall the recent article \cite{GT1} by Garofalo and Tralli, where nonlocal operators $(-\L)^s$ and their stationary counterparts are introduced. In particular, Hardy-Littlewood-Sobolev inequalities, Poincar\'e-type inequalities, and nonlocal isoperimetric inequalities are proved in \cite{GT2}, \cite{GT3}, and \cite{GT4}, respectively.

\medskip

This paper is organized as follows. Section \ref{sec2} is devoted to the study of Kolmogorov equations with constant coefficients and to the description of the underlying geometry for the study of this kind of equations. 
Section \ref{sec3} and Section \ref{sec4} are devoted to the analysis of Kolmogorov equations with H\"older continuous coefficients. In Section \ref{sec3} are presented the main known results concerning the fundamental solution, then the Cauchy and the Cauchy-Dirichlet problems are discussed. In Section \ref{sec4} mean value forumulas and Harnack type inequalities are described. This section also contains the detailed proof of a strong maximum principle for Kolmogorov operators in non divergence form.

\section{Kolmogorov operator with constant coefficients}
\label{sec2}
In the sequel of this section we consider the family of Kolmogorov operators of the form
\begin{equation} \label{cost-op}
\begin{split}
    \L := & \sum \limits_{i,j=1}^N a_{ij} \p^{2}_{x_i x_j} + \sum \limits_{i,j=1}^N b_{ij} x_i \p_{x_j} - \p_t \\
    = & \tr (A D^{2} u) + \langle B x , D u \rangle - \p_{t} u , 
    \qquad x \in \R^{N}, \, t \in \R, 
\end{split}
\end{equation}
where $A=(a_{ij})_{i,j=1,\ldots,N}$ and $B=(b_{ij})_{i,j=1,\ldots,N}$ are matrices with real constant coefficients, $A$ symmetric and non negative.
As we explained in Section \ref{sec1}, the fundamental solution to the degenerate equation \eqref{sdk} can be seen as the density of the solution to the stochastic differential equation \eqref{processo}. This is also the case when we consider a higher dimension. 
Specifically, let $\s$ be a $N \times m$ constant matrix, $B$ as in \eqref{cost-op}, and let $(W_{t})_{t \ge 0}$ be a $m$-dimensional Wiener process. Denote by $(X_{t})_{t \ge 0}$ the solution to the following $N$-dimensional Stochastic Differential Equation (SDE in short)
\begin{equation}
	\label{stoc}
	\begin{cases}
	d X_{t} = - B X_{t} \, dt + \s \, dW_{t} \\
	X_{t_{0}}=x_{0}.
	\end{cases}
\end{equation}
Then the \emph{backward Kolmogorov operator} $\mathcal{K}_{b}$ of $(X_{t})_{t \ge 0}$ acts on sufficiently regular functions $u$ as follows
\begin{align*} 
    \mathcal{K}_{b} u(y,s) = \p_s u(y,s) + \sum \limits_{i,j=1}^N  a_{ij} \p^{2}_{y_i y_j} u(y,s) - \sum \limits_{i,j=1}^N b_{ij} y_i \p_{y_j} u(y,s).
\end{align*}
where
\begin{equation} \label{a-def}
	A =\tfrac12 \s \s^{T},
\end{equation}
and the \emph{forward Kolmogorov operator} $\mathcal{K}_{f}$ of $(X_{t})_{t \ge 0}$ is the adjoint $\mathcal{K}_{b}^*$ of $\mathcal{K}_{b}$, that is
\begin{align*} 
    \mathcal{K}_{f} v(x,t)  = - \p_t v(x,t) + \sum \limits_{i,j=1}^N a_{ij} \p^{2}_{x_i x_j} v(x,t) + 
    \sum \limits_{i,j=1}^N b_{ij} x_i \p_{x_j} v(x,t) + \text{tr} (B) v(x,t),
\end{align*}
for sufficiently regular functions $v$. Note that $\mathcal{K}_f$ operator agrees with $\L$ in \eqref{cost-op} up to a multiplication of the solution by $\exp ( t \, \text{tr} (B))$. Also note that, because of \eqref{a-def}, it is natural to consider in \eqref{cost-op} a symmetric and non negative matrix $A$. When the matrix $A$ is strictly positive, the solution $(X_{t})_{t \ge 0}$ of the SDE \eqref{stoc} has a density $p=p(t-s,x,y)$ which is a solutions of the equations $\mathcal{K}_{b} p = 0$ and $\mathcal{K}_{f} p = 0$ in the following sense. For every $(x,t) \in \R^{N+1}$, the function $u(y,s) :=p(t-s,x,y)$ is a classical solution to the equation $\mathcal{K}_{b} u= 0$ in $\R^n \times ]- \infty, t[$ and, for every $(y,s) \in \R^{N+1}$, the function $v(x,t) = p(t-s,x,y)$ is a classical solution to $\mathcal{K}_{f}  v = 0$ in $\R^n \times ]s, + \infty[$. 
This is not always the case when $A$ is degenerate. In the sequel we give necessary and sufficient conditions on $A$ and $B$ for the existence of a density $p$ for the stochastic process $(X_{t})_{t \ge 0}$. These conditions are also necessary and sufficient for the hypoellipticity of $\L$. In order to state the afore mentioned conditions, we introduce some further notation. Following H\"ormander (see p. 148 in \cite{H}), we set, for every $t \in \R$,
\begin{equation} 	\label{c}
	E(t) = \text{exp} (-t B ), \qquad C(t) = \int_{0}^{t} E(s) \, A \, E^{T}(s) \, ds.
\end{equation}
The matrix $C(t)$ is symmetric and non-negative for every $t>0$, nevertheless it may occur that it is strictly positive. 
If this is the case, then $C(t)$ is invertible and the fundamental solution $\G(x_{0}, t_{0}; x,t)$ of $\L$ is
\begin{equation} \label{fun-sol}
    \G (x, t; \x, \t ) = \G ( x - E(t-\t)\x, t-\t),
\end{equation}
where $\G (x,t)=\G(x,t;0,0)$. Moreover, $\G(x,t) = 0$ for every $t \le 0$ and
\begin{equation} \label{fun-sol2}
    \G (x,t)=
    \frac{(4 \pi)^{-\frac{N}{2}}}{\sqrt{\text{det} C(t)}} \hspace{1mm}
            \text{exp} \left( - \frac{1}{4} \langle C^{-1} (t) x, x \rangle
            - t \hspace{1mm} \text{tr}(B) \right), \qquad t>0.
\end{equation}
The last notation we need to introduce allows us to write the operator $\L$ in the form \eqref{ol}. To do that, we recall that 
$\s = \left(\sigma_{jk}\right)_{j=1, \dots, N \atop k = 1, \dots, m}$ is a matrix with constant coefficients, and we set
\begin{equation} \label{e-fields}
 X_k := \frac{1}{\sqrt{2}} \sum_{j=1}^N \sigma_{jk} \p_{x_j}, \quad k= 1, \dots, m, 
 \qquad Y := \sum \limits_{i,j=1}^N b_{ij} x_i \p_{x_j} - \p_t.
\end{equation}
This allows us to rewrite the operator $\L$ in the form \eqref{ol} $\L = \sum_{j=1}^m X_{j}^{2} + Y$.
The following result holds true.

\begin{proposition}
	\label{equivalence}
Consider an operator $\L$ of the form \eqref{cost-op}, and let $\sigma$ be a $N \times m$ constant matrix such that 
$A$ writes as in \eqref{a-def}. Let $X_{1}, \ldots, X_{m}$, and $Y$ be the vector fields defined in \eqref{e-fields}.
Then the following statements are equivalent
\begin{enumerate}
	\item[{\rm \bf C1.}] (H\"ormander's condition):
	${\rm rank \, Lie} (X_{1}, \ldots, X_{m}, Y) (x,t) = N + 1$ for every $(x,t) \in \R^{N+1}$;  
	\item[{\rm \bf C2.}] $\ker (A)$ does not contain non-trivial subspaces which are invariant for $B$;
	\item[{\rm \bf C3.}] $C(t) > 0$ for every $t>0$, where $C(t)$ is defined in \eqref{c};
	\item[{\rm \bf C4.}] (Kalman's rank condition): 
	${\rm rank} \, \left( \s, B \s, \ldots, B^{N-1} \s \right) = N$;
	\item[{\rm \bf C5.}] for some basis of $\R^N$ the matrices $A$ and $B$ take the following block form 
		\begin{equation}
  				 \label{A}
    				A = \begin{pmatrix} 
      				  A_0 & \OO \\
       				  \OO & \OO
      				  \end{pmatrix}
			\end{equation}
    where $A_0$ is a symmetric strictly positive $m_0 \times m_0$ matrix, with $m_0 \le m$, and
			\begin{equation}
 			   \label{B}
 			   B = \begin{pmatrix}
 			       * &  * & \ldots & * & *  \\
 			       B_1 & * &  \ldots & * & * \\
     				   \OO & B_{2}  & \ldots & * & * \\
     				   \vdots & \vdots  & \ddots & \vdots & \vdots  \\
     				   \OO & \OO  & \ldots & B_{\k} & * 
   			 \end{pmatrix}
   			 =
      \begin{pmatrix}
         B_{0,0} &    B_{0,1}  & \ldots &     B_{0, \k - 1}  &   B_{0, \k }  \\  
        B_1   &    B_{1,1}  & \ldots &     B_{\k - 1 , 1}  &   B_{\k , 1}  \\
        \OO    &    B_2  & \ldots &   B_{\k - 1, 2}   &  B_{\k , 2}   \\
        \vdots & \vdots & \ddots & \vdots & \vdots \\
        \OO    &  \OO    &    \ldots & B_\k    &   B_{\k,\k}
    \end{pmatrix}
\end{equation}
where every block $B_j$ is a $m_{j} \times m_{j-1}$ matrix of rank $m_j$ with $j = 1, 2, \ldots, \k$. 
Moreover, the $m_j$s are positive integers such that 
		\begin{equation}
  			  \label{m-cond}
   			 m_0 \ge m_1 \ge \ldots \ge m_\k \ge 1, \quad
  			  {\rm and} \quad m_0 + m_1 + \ldots + m_\k = N
		\end{equation}
and the entries of the blocks denoted by $*$ are arbitrary. 
\end{enumerate}
When the above conditions are satisfied, then $\L$ is hypoelliptic, its fundamental solution $\Gamma$ defined in \eqref{fun-sol} and 
\eqref{fun-sol2}, is the density of the solution $(X_{t})_{t \ge 0}$ to \eqref{stoc}, and the problem \eqref{opti-pbm} is controllable.
\end{proposition}
The equivalence between {\rm \bf C1} and {\rm \bf C2} is proved by H\"ormander in \cite{H}. 
The equivalence between {\rm \bf C1}, {\rm \bf C2}, {\rm \bf C3}  and {\rm \bf C5} can be found in \cite{LP} (see Proposition A.1, 
and Proposition 2.1). The equivalence between {\rm \bf C3} and {\rm \bf C4} was first pointed out by Lunardi in \cite{LU}.
		
%

\begin{remark} \label{rem-1}
	The condition {\rm \bf C4} arises in control theory and it is related to the following controllability problem. For $x_0, x_1 \in \R^N$ and $t_0, t_1 \in \R$ with $t_0 < t_1$, find a ``control'' $\o \in L^1([t_0, t_1], \R^{m})$ such that
	\begin{equation} \label{opti-pbm}
		\begin{cases}
			\dot{x}(t) = - B x(t) + \s \o(t) ,\\
			x(t_{0}) = x_{0}, \quad x(t_{1}) = x_{1},
		\end{cases}
	\end{equation}
	where $\s$, $B$ are the same matrices appearing in \eqref{stoc}. It is known that a solution to the above control problem
	exists if, and only if, Kalmann's rank condition holds true (see \cite{ZAB}). 
\end{remark}
\begin{remark} \label{rem-2}
	We discuss the meaning of the matrix $C(t)$. 
	\begin{itemize}
	\item From the SDEs point of view, $2 C(t)$ is the covariance matrix of the solution $(X_{t})_{t\ge 0}$ to the SDE \eqref{stoc}. In general, $(X_{t})_{t\ge 0}$ is a Gaussian process and its density $p$ is defined on $\R^{N}$ when its covariance matrix is positive definite. If this is not the case, the trajectories of $(X_{t})_{t\ge 0}$ belong to a proper subspace of $\R^{N}$. 
	\item The matrix $C(t)$ has a meaning also for the optimal control point of view. Indeed, it is known that 
	\begin{equation*}
		\langle C(t-t_{0})^{-1} \left(x - E(t-t_{0})x_{0}\right) , 
		x - E(t-t_{0})x_{0} \rangle = \inf \int_{t_{0}}^{t} \abs{\w(s)}^{2} \, ds,
	\end{equation*} 
	where the infimum is taken in the set of all controls for \eqref{opti-pbm} (see \cite{LeeMarkus},
	Theorem 3, p. 180). In particular, when $(x_{0}, t_{0}) = (0,0)$ the optimal 
	cost is $\langle C(t)^{-1} x , x \rangle$, a quantity that appears in the expression for the fundamental solution 
	$\G$ in \eqref{fun-sol2}. As we will see in the sequel, this fact will be used to prove asymptotic bounds for positive solutions to Kolmogorov equations (see \eqref{upper-bound} in Theorem \ref{regularity-fun}).
	\end{itemize}
In view of the above assertions, the equivalence of {\rm \bf C3} and {\rm \bf C4} can be interpreted as follows. A control $\o \in L^1([t_0, t_1], \R^{m})$ for the problem \eqref{opti-pbm} exists if, and only if, the trajectories of the Stocastic Process $(X_{t})_{t\ge 0}$ reach every point of $\R^N$. 
\end{remark}

\subsection{Lie Group}
In this Section we focus on the non-Euclidean invariant structure for Kolmogorov operators of the form \eqref{cost-op}. 
This non commutative structure was first used by Garofalo and E. Lanconelli in \cite{GL}, then explicitly written and 
thoroughly studied by E. Lanconelli and Polidoro in \cite{LP}. 
Here and in the sequel we denote by $\lL$, the family of Kolmogorov operators $\L$ satisfying the equivalent conditions of 
Proposition \ref{equivalence}. We also assume the basis of $\R^{N}$ is such that the \emph{constant} matrices $A$ and $B$ have the form \eqref{A} and \eqref{B}, respectively. 

We now define a non commutative algebraic structure on $\R^{N+1}$ introduced in \cite{LP}, that replaces the Euclidean one in the study of Kolmogorov operators. 

\medskip

\noindent 
{\sc Lie group.} \textit{Consider an operator $\L$ in the form \eqref{cost-op} and recall the notation \eqref{c}. Let}
\begin{equation}
    \label{law}
    \GG= ( \R^{N+1}, \circ), \qquad (x,t) \circ (\x, \t) = ( \x + E(\t) x, \hspace{1mm} t + \t). 
\end{equation}
\textit{Then $\GG$ is a group with zero element $(0,0)$, and inverse}
\begin{equation}
    \label{nverse}
    (x, t)^{-1} := (-E(-t) x, \, -t). 
\end{equation}
\textit{For a given $\z \in \R^{N+1}$, we denote by $\ell_{\z}$ the left traslation defined as}
\begin{equation*}
	\ell_{\z}: \R^{N+1} \rightarrow \R^{N+1}, \quad \ell_{\z} (z) = \z \circ z.
\end{equation*}
\textit{Then the operator $\L$ is left invariant with respect to the Lie product $\circ$, that is}
\begin{equation}
	\label{ell}
    \L \circ \ell_{\z} = \ell_{\z} \circ \L \qquad {\rm \textit{or, equivalently,}} 
    \qquad \L\left( u( \z \circ z) \right)  = \left( \L u \right) \left( \z \circ z \right),
\end{equation}
\textit{for every $u$ sufficiently smooth.}

\medskip

We omit the details of the proof of the above statements as they are elementary. We remark that, even though we are interested in hypoelliptic operators $\L$, the definition of the Lie product $\circ$ is well posed wether or not we assume the H\"ormander's condition. 
Also note that 
\begin{equation}
    \label{prodotto}
    (\x,\t)^{-1} \circ (x,t) =  ( x - E(t-\t) \x, t - \t) , \qquad
    (x,t), (\x, \t) \in \R^{N+1},
\end{equation}
then the meaning of \eqref{fun-sol} can be interpreted as follows: 
\begin{equation} \label{fun-sol3}
 \G(x,t;\xi, \tau) = \Gamma \big((\x,\t)^{-1} \circ (x,t)\big).
\end{equation}

Among the class of Kolmogorov operators $\lL$, the invariant operators with respect to a certain family of dilations 
$(D(r))_{r>0}$ play a central role. We say that $\L \in \lL$ is invariant with respect to $(D(r))_{r>0}$ if
\begin{equation}
	\label{Ginv}
      	 \L \left( u \circ D(r) \right) = r^2 D(r) \left( \L u \right), \quad \text{for every} \quad r>0,
\end{equation}
for every function $u$ sufficiently smooth. This property can be read in the expression of the matrix $B$ (see Proposition 2.2 of \cite{LP}).
\begin{proposition}
	\label{prop22}
	Let $\L$ be an operator of the family $\lL$. Then $\L$ satisfies \eqref{Ginv} if, and only if, 
the matrix $B$ as this form 
\begin{equation}
    \label{B0}
    B_{0} = \begin{pmatrix}
        \OO &  \OO & \ldots & \OO & \OO  \\
        B_1 & \OO &  \ldots & \OO & \OO \\
        \OO & B_{2}  & \ldots & \OO & \OO \\
        \vdots & \vdots  & \ddots & \vdots & \vdots  \\
        \OO & \OO  & \ldots & B_{\k} & \OO 
    \end{pmatrix}.
\end{equation}
In this case
	\begin{equation}
		\label{fam-dil}
   		 D(r) = \diag ( r \I_{m_0} , r^3 \I_{m_1}, \ldots, r^{2\k+1} \I_{m_\k}, r^2 ) 
   		 \quad {\rm for \, every} \, \, r > 0,
	\end{equation}
where $\I_{m_j}$ denotes the identity matrix in $\R^{m_j}$. In the sequel we denote by $\lL_{0}$ the family of dilation-inviariant operators belonging to $\lL$. 
\end{proposition}
It is useful to denote by $\left(D_{0}(r)\right)_{r > 0}$ the family of spatial dilations defined as 
\begin{equation}
	\label{fam-dil-space}
	D_{0}(r) = \text{diag} ( r \I_{m_0} , r^3 \I_{m_1}, \ldots, r^{2\k+1} \I_{m_\k} ) 
   		 \quad {\rm for \, every} \, \, r > 0.
\end{equation}

\medskip

\noindent 
{\sc Homogeneous Lie group.} \textit{If the matrix $B$ has the form \eqref{B0}, we say that the following structure}
\begin{equation}
    \label{hom-group}
    \GG_0= \left( \R^{N+1}, \circ, \left( D(r) \right)_{r>0}\right)
\end{equation}
\textit{is a homogeneous Lie group. In this case, because $D_{0}(r) \, E(t) \, D_{0}(r) \, = \, E(r^{2} t)$ is verified when 
$B$ has the form \eqref{B0}, the following distributive property holds}
\begin{equation}
	\label{compatible}
	D(r)(\z \circ z ) = (D(r) \z ) \circ (D(r) z), \qquad D(r)(z^{-1}) = (D(r)z)^{-1}.
\end{equation}

\medskip

\begin{remark}
	A measurable function $u$ on $\GG_{0}$ will be called homogeneous of degree $\a \in \R$  if
	\begin{equation*}
		u(D_{r}(z)) = r^{\a} u(z) \qquad \text{for every } \, z \in \R^{N+1}.
	\end{equation*}
	A differential operator $X$ will be called homogeneous of degree $\b \in \R$ with respect to $(D_{r})_{r \ge 0}$ if 
	\begin{equation*}
		X u (D_{r}(z)) = r^{\b} \left( X u \right) (D_{r}(z) ) \qquad \text{for every } \, z \in \R^{N+1},
	\end{equation*}
	and for every sufficiently smooth function $u$. Note that, if $u$ is homogeneous of degree $\a$ and $X$ is homogeneous of degree $\b$, then $X u$ is homogeneous of degree $\a - \b$. 
	
	As far as we are concerned with the vector fields of the Kolmogorov operators as defined in \eqref{e-fields}, we
	have that $X_{1}, \ldots, X_{m}$ are homogeneous of degree $1$ and $Y$ is homogeneous of degree $2$
	with respect to $(D_{r})_{r \ge 0}$. In particular, $\L = \sum_{j=1}^m X_j + Y$ is is homogeneous of degree $2$.
\end{remark}

\begin{remark}
 The 
 presence of the exponents $1, 3, \dots, 2 \k + 1$ in the matrix $D$ can be explained as follows. The usual parabolic dilation in the first $m_0$ coordinates of $\R^N$ and in time is due to the fact that $\L$ is non degenerate with respect to $x_1, \dots, x_{m_0}$. The remaining coordinates appear as we check the H\"ormander's condition. 
For instance, consider the Kolmogorov operator
\begin{equation*}
	\L = \p_{x_{1}}^{2} + x_{1} \p_{x_{2}} + x_{2} \p_{x_{3}} - \p_{t} = X_{1}^{2} + Y.
\end{equation*}
To satisfy the H\"ormander condition we need $\k = 2$ commutators $\p_{x_{2}}= [X_1,Y]= X_{1} Y - Y X_{1}$ and $\p_{x_{3}} = [[X_{1}, Y], Y]$. Because $Y$ needs to be considered as a second order derivative, we have that $\p_{x_{2}}$ and $\p_{x_{3}}$ are derivatives of order $3$ and $5$, respectively.
On the other hand, the matrices $A$, $B$ and $D_{0}(r)$ associated to this operator are
\begin{align*}
	A = \begin{pmatrix}
		1 & 0 & 0 \\
		0 & 0 & 0 \\
		0 & 0 & 0 \\
	\end{pmatrix},
	\qquad
	B = \begin{pmatrix}
		0 & 0 & 0 \\
		1 & 0 & 0 \\
		0 & 1 & 0 \\
	\end{pmatrix},
	\qquad
	D_{0}(r) = \begin{pmatrix}
		r & 0 & 0 \\
		0 & r^{3} & 0 \\
		0 & 0 & r^{5} \\
	\end{pmatrix}.
\end{align*}
The same argument can be applied to operators that need $\k>2$ steps to satisfy H\"ormander's rank condition.
\end{remark}

The integer numbers
\begin{equation}
	\label{hom-dim}
	Q := m_{0} + 3m_{1} + \ldots + (2\k+1) m_{k},  \quad \text{and} \quad Q + 2
\end{equation}
will be named \emph{homogeneous dimension of $\R^{N}$ with respect to $(D_0(r))_{r>0}$}, and \emph{homogeneous dimension of $\R^{N+1}$ with respect to $(D(r))_{r >0}$}, because we have that
\begin{equation*}
	\det \, D_0(r) = r^{Q } \quad \text{and} \quad \det \, D(r) = r^{Q + 2} \qquad \text{for every} \ r > 0.
\end{equation*}
We now introduce a homogeneous semi-norm of degree $1$ with respect to the family of dilations $(D(r))_{r>0}$ and a quasi-distance which is invariant with respect to the group operation $\circ$. 
\begin{definition}
	\label{norm-def}
    For every $z=(x,t) \in \R^{N+1}$ we set
    \begin{equation}
    	\label{bnorm}
    	\bnorm{z} = |t|^{\frac12} + \babs{x}, \qquad \babs{x} = \sum_{j=1}^{N} |x_{j}|^{\frac{1}{q_{j}}},
    \end{equation} 
    where the numbers $q_{j}$ are associated to the dilation group $(D(r))_{r>0}$ as follows
    \begin{equation*}
        D(r) = \diag \left( r^{q_1}, \ldots, r^{q_N} , r^{2} 	\right) .
    \end{equation*}
\end{definition}

	The semi-norm $\bnorm{\cdot}$ is homogeneous of degree $1$, that is
	\begin{equation*}
		\bnorm{D(r) z} = r \bnorm{z} \qquad {\rm for \; every} \; r > 0, z \in \R^{N+1}.
	\end{equation*}
    Because every norm is equivalent to any other in $\R^{N+1}$, other definitions have been used in the literature. For instance in \cite{M} it is chosen the following one. For every $z=(x_{1}, \ldots, x_{N}, t) \in \R^{N+1} \setminus \{ 0 \}$ the norm of $z$ is the unique positive solution $r$ to the following equation
    \begin{equation}
    	\label{eqsemi}
        \frac{x_1^{q_1}}{r^{2 q_1}} + \frac{x_2^{q_2}}{r^{2 q_2}} + \ldots +
        \frac{x_N^{q_N}}{r^{2 q_N}} + \frac{t^2}{r^4} = 1.
    \end{equation}
Note that, if we choose \eqref{eqsemi}, the set $\big\{ z \in \R^{N+1} : \|z \| = r \big\}$ is a smooth manifold for every positive $r$, which is note the case for \eqref{bnorm}. 

\medskip

Based on Definition \ref{norm-def}, in the following we introduce a quasi-distance 
$d: \R^{N+1} \times \R^{N+1} \rightarrow [0, + \infty [$ (see Definition \ref{dist-def} below). 
This means that:
	\begin{enumerate}
		\item $d(z,w) = 0$ if and only if $z=w$ for every $z, w \in \R^{N+1}$;
		\item for every compact subset $K$ of $\R^{N+1}$, there exists  a positive constant $C_{K} \ge 1$
such that
			\begin{equation}
				\label{triangle}
\begin{split}
 & d(z,w) \le C_{K} d(w,z); \\ 
 & d(z,w) \le C_{K} \left( d(z, \z) + d(\z, w) \right), \quad \text{for every } \, z, w, \z \in K.
\end{split}
\end{equation}
\end{enumerate} 
The proof of \eqref{triangle} is given in Lemma 2.1 of \cite{PODF}. Definition \ref{dist-def} is given for general non-homogeneous Lie groups. This requires the notion of \emph{principal part operator} discussed in the next section.
We point out that the constant $C_{K}$ doesn't depend on $K$ in the case of homogeneous groups (see Proposition 2.1 in \cite{M}).

\subsection{Principal part operator}
In the last part of this Section we show that the dilation invariant operators are the blow-up limit of the operator belonging to $\lL$. In order to identify the appropriate dilation, we denote by $\L_{0}$ the \textit{principal part operator of $\L$} obtained from \eqref{cost-op} by substituting the matrix $B$ with $B_{0}$ as defined in \eqref{B0}, that is
\begin{equation}
	\label{principal}
	\L_{0} = \div (AD) + \langle B_{0} x, D \rangle - \p_{t}.
\end{equation}
Since $\L_{0}$ is dilation-invariant with respect to $(D(r))_{r>0}$, we define $\L_{r}$ as the \emph{scaled operator} of $\L$ in terms of $(D(r))_{r>0}$ as follows
\begin{equation}
	\label{Lr}
	\L_{r} := r^{2} \, D(r) \, \circ \L \circ D \left( 1/r \right) = \tr(A D^2 ) + \langle B_{r} x , D \rangle - \p_{t},
\end{equation}
where $B_{r}= D(r) \, B D \left( 1/r \right)$ is given by 
\begin{equation}
    \label{Bstar}
    B_r = 
    \begin{pmatrix}
        r^2 B_{0,0} &   r^4 B_{0,1}  & \ldots &    r^{2\k} B_{0, \k - 1}  &  r^{2\k+2} B_{0, \k }  \\  
        B_1   &   r^2 B_{1,1}  & \ldots &    r^{2\k-2} B_{\k - 1 , 1}  &   r^{2\k} B_{\k , 1}  \\
        \OO    &    B_2  & \ldots &  r^{2\k-4} B_{\k - 1, 2}   &  r^{2\k-2} B_{\k , 2}   \\
        \vdots & \vdots & \ddots & \vdots & \vdots \\
        \OO    &  \OO    &    \ldots & B_\k    &   r^2 B_{\k,\k}
    \end{pmatrix}.
\end{equation}
Clearly $\L_{r}=\L$ for every $r>0$ if and only if $B = B_{0}$, and the principal part $\L_0$ of $\L$ is obtained as the limit of \eqref{Lr} as $r \rightarrow 0$. 

\medskip

The invariance structures of the operator $\L$ also reveal themselves in the expression of the fundamental solution $\G$. In particular, as noticed above, $\G$ is translation invariant, as it satisfies the identity \eqref{fun-sol3}.
As far as we are concerned with the dilation invariance, the fundamental solution $\G_{0}$ of $\L_{0}$ is a homogeneous function of degree $-Q$ with respect to the  dilation $(D(r))_{r>0}$, that is
	\begin{equation} \label{Gamma0-hom}
    		\G_0 (D(r) z) = r^{-Q} \G_0(z) \qquad {\rm for \, every} \quad z \in \R^{N+1} \setminus \{ 0 \}, \, r > 0,
	\end{equation}
where $Q$ is the spatial homogeneous dimension of $\R^{N+1}$ introduced in \eqref{hom-dim}. Moreover, the expression of $\G_0$ writes in terms of $D_{0}(r)$. Indeed, the matrix $C(t)$ defined in \eqref{c} satisfies the following identity
	\begin{equation*}
		C(t) = D_{0}(\sqrt{t}) \, C(1) \, D_{0}(\sqrt{t}) \qquad {\rm for \, every} \, t > 0,
	\end{equation*}
and 
	\begin{equation*}
		\label{gamma-rappr}
		\G_0(x,t) = \frac{C_{N}}{t^{\frac{Q}{2}}} \exp \left( - \tfrac14 \langle C^{-1}(1) \, D_{0} \left( \tfrac{1}{\sqrt{t}} \right) x, D_{0} \left( \tfrac{1}{\sqrt{t}} \right) x \rangle \right),
	\end{equation*}
	where $C_{N}$ is the positive constant 
	\begin{equation*}
		\label{cn}
		C_{N} = (4 \pi)^{-\frac{N}{2}} (\det \, C(1) )^{-\frac12}.
	\end{equation*}
We refer to \cite{LP}, \cite{K3}, \cite{K5} for the proof of the above statements. Eventually, Theorem 3.1 in \cite{LP} provides us with a quantitative comparison between $\G$ and $\G_{0}$. 
\begin{theorem}
	\label{teo31}
	Let $\L$ be an operator of the class $\lL$ and let $\L_{0}$ be its principal part as defined in \eqref{principal}.
	Then for every $K>0$ there exists a positive constant $\varepsilon >0$ such that
	\begin{equation}
		\label{bounds}
		(1-\e) \G_{0} (z) \le \G(z) \le (1+\e) \G_{0}(z)
	\end{equation}
	for every $z \in \R^{N+1}$ such that $\G_{0}(z) \ge K$. Moreover, $\e =\e(K) \rightarrow 0$ as $K \rightarrow + \infty$.
\end{theorem}
Note that the above result doesn't hold true in the set $\big\{\G_{0} < K\big\}$ (see formula (1.30) in \cite{LP}).

\medskip

We now introduce the quasi-distance $d$ for a generic Lie group $\GG$. In the following definition
``$\circ$'' denotes the traslation of $\L$, and the norm $\bnorm{\cdot}$ is the one associated to $\L_{0}$. 
\begin{definition}
	\label{dist-def}
	For every $z, w \in \R^{N+1}$, we define a quasi-distance
	 $d(z,w)$ invariant with respect to the translation
	group $\GG_{0}$ as follows
  	   \begin{equation}
   	 	\label{quasi-distance}
    		d(z, w) = \bnorm{ z^{-1} \circ w },
          \end{equation}
	and we denote by $B_r(z)$ the $d-$ball of center $z$ and radius $r$. 
\end{definition}

\begin{definition}
    \label{holdercontinuous}
    Let $\a$ be a positive constant, $\a \le 1$, and let $\O$ be an open subset of $\R^{N+1}$. We 
    say a function $f : \O \longrightarrow \R$ is H\"older continuous with exponent $\a$ in $\O$
    with respect to the groups $\GG=(\R^{N+1}, \circ)$ and $(D(r))_{r>0}$ (in short: H\"older 
    continuous with exponent $\a$, $f \in C^\a (\O)$) if there exists a positive constant $k>0$ such that 
    \begin{equation*}
        | f(z) - f(\z) | \le k \; d(z,\zeta)^{\a} \qquad { \rm for \, every \, } z, \z \in \O.
    \end{equation*}
   To every bounded function $f \in C^\a (\O)$ we associate the norm
    	\begin{equation*}
       		 |f|_{\a, \O} \hspace{1mm} = \hspace{1mm} \sup \limits_\O |f| \hspace{1mm} + \hspace{1mm} 
       		 \sup \limits_{z, \z \in \O \atop  z \ne \z} \frac{|f(z) - f(\z)|}{d(z,\zeta)^{\a}}.
        \end{equation*}
    Moreover, we say a function $f$ is locally H\"older continuous, and we write $f \in C^{\a}_{\loc}(\O)$,
    if $f \in C^{\a}(\O')$ for every compact subset $\O'$ of $\O$.
\end{definition}

\begin{remark}
   Let $\O$ be a bounded subset of $\R^{N+1}$.
   If $f$ is a H\"older continuous function of exponent $\a$ in the usual Euclidean sense, 
    then $f$ is H\"older continuous of exponent $\a$. Vice versa, if $f \in C^\a (\O)$ then $f$
    is a $\b-$H\"older continuous in the Euclidean sense, where $\b = \tfrac{\a}{2\k+1}$
    and $\k$ is the constant appearing in \eqref{B}.
\end{remark}

\section{Kolmogorov operator with H\"older continuous coefficients}
\label{sec3}
In this section we consider Kolmogorov operator in non-divergence form in $\R^{N+1}$
\begin{equation}
    \label{mmL}
    \L = \sum \limits_{i,j=1}^{m_0} a_{ij}(x,t) \p^2_{x_ix_j} \, + \, 
    \sum \limits_{j=1}^{m_0} b_{j}(x,t) \p_{x_j} \, + \,
    \langle B x, D \rangle - \, \p_t, \qquad 
    {\rm for} \, (x,t) \in \R^{N+1}
\end{equation}
with continuous coeficients $a_{ij}$'s and $b_{j}$'s. As in the parabolic case, the classical theory for degenerate Kolmogorov operators is developed for spaces of H\"older continuous functions introduced in Definition \ref{holdercontinuous}. We remark that this definition relies on the Lie group $\GG$ \eqref{law}, that is an invariant structure for the \emph{constant coefficients} operators. Even though the \emph{non-constant coefficients} operators in \eqref{mmL} are not invariant with respect to $\GG$, we will rely on the Lie group invariance of the model operator 
\begin{equation}
    \label{mmLap}
    \Delta_{m_0} + Y = \sum \limits_{j=1}^{m_0}  \p^2_{x_j} \, + \, \langle B x, D \rangle - \, \p_t, 
\end{equation}
associated to $\L$. Indeed, this is a standard procedure in the study of uniformly parabolic operators. 
We next list the standing assumptions of this section:
\begin{itemize}
   \item[\textbf{(H1)}] $B= (b_{i,j})$ is a $N \times N$ real constant matrix of the type \eqref{B}, with blocks $B_{j}$ of rank $m_{j}$ 
   	and $*-$blocks arbitrary;
   \item[\textbf{(H2)}] $A=(a_{ij}(z))_{i,j=1, \ldots, m_0}$ is a symmetric matrix of the form \eqref{A}, i.e. $a_{ij}(z)= a_{j,i}(z)$ for $i,j=1, \ldots, 
   	m_{0}$, with $1 \le m_{0} \le N$. Moreover, it is positive definite in $\R^{m_0}$ and there exist a positive constant $\l$ such that 
	\begin{equation*}
  	  \frac{1}{\l} \sum \limits_{i=1}^{m_0} \abs{\x_i}^2 \le \sum \limits_{i,j=1}^{m_0} a_{ij}(z) \x_i \x_j 
  	  \le \l \sum \limits_{i=1}^{m_0} \abs{\x_i}^2 	
	\end{equation*}
	for every $(\x_1, \ldots, \x_{m_0}) \in \R^{m_0}$ and $z \in \R^{N+1}$;
	\item[\textbf{(H3)}] there exist $0 < \a \le 1$ and $M > 0$ such that 
	\begin{equation*}
  	     \abs{ a_{ij}(z) - a_{ij}(\z)} \le M \, d(z,\zeta)^{\a}, \qquad 
  	     \abs{ b_{j}(z) - b_{j}(\z)} \le M \, d(z,\zeta)^{\a},
	\end{equation*}
for every $z, \z \in \R^{N+1}$ and for every $i,j = 1, \ldots, m_{0}$,
\end{itemize}
Note that, if $m_0 = N$, the operator $\L$ is uniformly parabolic and $B = \OO$. In particular the model operator \eqref{mmLap} is the heat equation and we have that $d\big( (\x,\t), (x,t)\big) = |\x-x|+|\t-t|^{1/2}$, so that we are considering the parabolic modulus of continuity. 

In the sequel we refer to the Assumption \textbf{(H3)} by saying that the coefficients $a_{ij}$'s and $b_{j}$'s belong to the space $C^\alpha$ introduced in Definition \ref{holdercontinuous}. We next give the definion of classic solution to the equation $\L u = f$ under minimal regularity assumptions on $u$. A function $u$ is Lie differentiable with respect to the vector field $Y$ defined in \eqref{e-fields} at the point $z=(x,t)$ if there exists and is finite 
\begin{equation} \label{lie-diff}
		Yu(z) := \lim \limits_{s \rightarrow 0} \frac{u(\g(s)) - u(\g(0))}{s}, \qquad \g(s) = (E(-s) x, t - s ).
\end{equation}
Note that $\g$ is the integral curve of $Y$ from $z$. Clearly, if $u \in C^{1}(\O)$, with $\O$ open subset of $\R^{N+1}$, then $Y u (x,t)$ agrees with $\langle B x, D u(x,t) \rangle - \p_{t} u (x,t)$ considered as a linear combination of the derivatives of $u$. 
\begin{definition}
	\label{solution}
	A function $u$ is a solution to the equation $\L u = f$ in a domain $\O$ of $\R^{N+1}$
	if there exists the Euclidean derivatives $\p_{x_{i}} u, \p_{x_{i},x_{j}} u \in C(\O)$
	for $i,j = 1, \ldots, m_{0}$, the Lie derivative $Yu \in C(\O)$, and the equation
	\begin{equation*}
	       \sum \limits_{i,j=1}^{m_0} a_{ij}(z) \p^2_{x_ix_j} u(z)
	       + \sum \limits_{j=1}^{m_0} b_{j}(z) \p_{x_j} u(z) 
	       \, + Y u(z) = f(z)
	\end{equation*} 
	is satisfied at any point $z=(x,t) \in \O$.
\end{definition}

The natural functional setting for the study of classical solutions is the space 
 \begin{equation}
 	\label{c2alfa}
	C^{2,\a} (\O) = \left\{ u \in C^{\a} (\O) \; \mid \; \p_{x_{i}}u , \p^2_{x_{i} x_{j}} u, Y u \in C^{\a} (\O), \quad 
	\text{for } i,j = 1, \ldots, m_{0}  \right\},
   \end{equation}
where $C^{\a} (\O)$ is given in Definition \ref{holdercontinuous}. Moreover, if $u \in C^{2,\a} (\O)$ then we define the norm
\begin{equation}
 	\label{c2alfa-norma}
	| u |_{2 + \a, \O } := | u |_{\a, \O } \; + \; \sum \limits_{i=1}^{m_{0}} | \p_{x_{i}}u |_{\a, \O } \; + \;  \sum \limits_{i,j=1}^{m_{0}} |\p^2_{x_{i} x_{j}}u |_{\a, \O }  \; + \;  |Y u |_{\a, \O }.
   \end{equation}
Clearly, the definition of $C^{2,\a}_{\loc} (\O)$ follows straightforwardly from the definition of $C^{\a}_{\loc} (\O)$. 
A definition of the space $C^{k,\a} (\O)$ for every positive integer $k$ is given and discussed in the work \cite{PPP} by Pagliarani, Pascucci and Pignotti, where a proof of the Taylor expansion for $C^{k,\a} (\O)$ functions is given. It is worth noting that the authors of \cite{PPP} require weaker regularity assumptions for the definition of the space $C^{2,\a}$ than the ones considered here in \eqref{c2alfa}.

As in the uniformly elliptic and parabolic case, fundamental results in the classical regularity theory are the Schauder estimates.
We recall that Schauder estimates for the dilation invariant Kolmogorov operator (i.e. where the matrix $B=B_{0}$) with H\"older continuous 
coefficients were proved by M. Manfredini in \cite{M} (see Theorem 1.4). 
Manfredini result was then extended by Di Francesco and Polidoro in \cite{PODF} to the non-dilation invariant
case.
\begin{theorem}
    \label{schauder1}
    Let us consider an operator $\L$ of the type \eqref{mmL} satisfying assumptions \textbf{(H1)}, \textbf{(H2)}, \textbf{(H3)}
    with $\a < 1$. 
    Let $\O$ be an open subset of $\R^{N+1}$, $f \in C^{\a}_{\loc}(\O)$ and let $u$ be a classical solution to $\L u = f$ in $\O$. 
    Then for every $\O^{'} \subset \subset \O^{''} \subset \subset \O$ there exists a positive constant $C$ such that
    \begin{equation*}
        \label{scauder2}
        | u |_{2 + \a, \O^{'} } \le C \Big( \sup \nolimits_{\O^{''}} |u| \; + \; |f|_{\a, \O^{''} } \Big).
    \end{equation*}
\end{theorem}
A more precise estimate taking into account the distance between the point and the boundary of the set $\O$ can be found in \cite{M} (see Theorem 1.4) for the dilation invariant case. We omit here this precise statement because it requires the introduction of further notation. 
We also recall that analogous Schauder estimates have been proved by several authors in the framework of semigroup theory, where they consider solutions which are not classical in the sense of Definition \ref{solution}. Among others, we refer to Lunardi \cite{LU}, Lorenzi \cite{LO}, Priola \cite{PR}, Delarue and Menozzi \cite{DM}.

\subsection{Fundamental Solution and Cauchy Problem}
The existence of a fundamental solution $\G$ for the operator $\L$ satisfying the assumptions \textbf{(H1)}, \textbf{(H2)} and \textbf{(H3)} has been proved using the Levi's parametrix method. The first results of this type are due to M. Weber \cite{Weber}, to Il'In \cite{IlIn} and to Sonin \cite{Sonin} who assumed an Euclidean regularity on the coeficients $a_{ij}$'s and $b_{j}$'s. Later on, Polidoro applied in \cite{P2} the Levi parametrix method for the \emph{dilation inviariant} operator $\L$ (i.e. under the additional assumption that $B$ has the form \eqref{B0}), then Di Francesco and Pascucci removed this last assumption in \cite{PDF}. 

The Levi's parametrix method is a constructive argument to prove existence and bounds of the fundamental solution. For every $\z \in \R^{N+1}$, the parametrix $Z (\, \cdot \, , \z)$ is the fundamental solution, with pole at $\z$, of the following operator
\begin{equation} \label{ellez}
	\L_{\z} = \sum \limits_{i,j=1}^{m_{0}} a_{ij}(\z) \, \p^{2}_{x_{i} x_{j}} \, + \, \langle B x, D \rangle \, - \, \p_{t}.
\end{equation}
The method is based on the fact that, if the coeficients $a_{ij}$'s are continuous and the coefficiens $b_{j}$'s are bounded, then $Z$ is a good approximation of the fundamental solution of $\L$, because 
\begin{equation*}
	\L Z(z,\z) = \sum \limits_{i,j=1}^{m_{0}} \left( a_{ij}(z) - a_{ij}(\z) \right) \, \p^2_{x_{i} x_{j}} Z (z,\z) + 
	\sum \limits_{j=1}^{m_{0}} b_{j}(z) \, \p_{x_{j}} Z (z,\z),
\end{equation*}
at least as $z$ is close to the pole $\z$. We look for the fundamental solution $\Gamma$ as a solution of the following Volterra equation
\begin{equation} \label{parametrix}
	\G (x,t,\x,\t) = Z(x,t,\x,\t) + \int_{\t}^t \int_{\R^N} Z(x,t, y,s) G(y,s,\x,\t) dy \, ds,
\end{equation} 
where the unknown function $G$ is obtained by a fixed point argument. It turns out that
\begin{equation} \label{parametrix-G}
	G(z,\z) = \sum_{k=1}^{+ \infty} (\L Z)_k (z,\z),
\end{equation}
where $(\L Z)_1 (z,\z) = \L Z (z,\z)$ and, for every $k \in \N$,  
\begin{equation*}
	(\L Z)_{k+1} (x,t,\x,\t) = \int_{\t}^t \int_{\R^N} \L Z(x,t, y,s) (\L Z)_k (y,s,\x,\t) dy \, ds.
\end{equation*}
Let's point out that $Z$ is explicitly known by formulas \eqref{fun-sol} and \eqref{fun-sol2}, then the equations \eqref{parametrix} and \eqref{parametrix-G} give explicit bounds for $\G$ and for its derivatives (see equations \eqref{upper-bound} and \eqref{stime} below). 
We summarize here the main results of the articles \cite{P2} and \cite{PDF} on the existence and bounds for the fundamental solution. 
%
\begin{theorem}
	\label{regularity-fun}
	Let $\L$ be an operator of the form \eqref{mmL} under the assumptions \textbf{(H1)}, \textbf{(H2)}, \textbf{(H3)}. Then there exists a fundamental solution $\G(\cdot, \z)$ to $\L$ with pole at $\z \in \R^{N+1}$ such that: 
	 \begin{enumerate}
	 	\item $\G(\cdot, \z) \in L^{1}_{\loc} (\R^{N+1}) \cap C(\R^{N+1} \setminus
		\{ \z \} )$;
		\item for every $\phi \in C_b(\R^N)$ the function
$$
  u(x,t)=\int_{\R^{N}}\Gamma(x,t;\x,0)\phi(\x)d \x, 
$$
is a classical solution of the Cauchy problem
\begin{equation} \label{cauchyproblem}
\left\{
  \begin{array}{ll}
    \L u = 0, & \hbox{$(x,t) \in \R^{N} \times \R^{+}$} \\
    u(x,0)=\varphi (x) & \hbox{$(x,t) \in \mathbb{R}^N$.}
  \end{array}
\right.
\end{equation}
	\item  For every $(x, t), (\x, \t) \in \R^{N+1}$ such that $\t < t$ we have that
		\begin{equation*}
 			\int \limits_{\R^{N}} \G(x,t,\x, \t) \; d \x = 1;
 		\end{equation*} 
	\item the reproduction property holds for every $(y,s) \in \R^{N+1}$ with $\t < s < t$: 
		\begin{equation*}
			\G(x,t,\x,\t) = \int \limits_{\R^{N}} \G(x,t,y,s) \, \G(y,s,\x,\t) \, dy;
		\end{equation*}
	\item for every positive $T$ and for every $\Lambda > \l$, with $\l$ as in \textbf{\emph{(H1)}},
	 there exists a positive constant $c^{+}= c^{+}(\Lambda, \l,T)$ such that 
		\begin{equation}
	 	\label{upper-bound}
		c^{-} \, \G^{-} (z, \z) \le \G(z, \z) \le c^{+} \, \G^{+} (z, \z) \qquad {\rm for \, every}\, z, \z \in \R^{N+1}, \, 0 < t- \t 
		< T,
	 \end{equation}
for every $(x,t), (\x, \t) \in \R^{N+1}$ with $0 < t- \t < T$. Here, 
 $\G^{+}$ and $\G^{-}$ are, respectively, the fundamental solutions of the following operators:
	 \begin{align*}
	 	\L^{+} = \l \D_{m_{0}} + \langle B x, D \rangle - \p_{t}  \quad {\rm and} \quad
		\L^{-} = \l^{-1} \D_{m_{0}} + \langle B x, D \rangle - \p_{t}.
	 \end{align*}
	\end{enumerate}
\end{theorem}
Once the uniqueness of the Cauchy problem is guaranteed, points 3. and 4. of the above theorem will follow from point 2. The lower bound in \eqref{upper-bound} is proved by using the Harnack inequality presented in Theorem \ref{harnack-podf} and following the technique introduced by Aronson and Serrin \cite{AS} 
for the classic parabolic case. We remark that property 3. of Theorem \ref{regularity-fun} doesn't hold unless we require further regularity assumptions on the coefficients $a_{ij}$'s and $b_{j}$'s needed to define the formal adjoint $\L^{*}$ of $\L$. 

\medskip 

In view of \eqref{cauchyproblem}, the fundamental solution is the most natural tool to deal with the Cauchy problem
associated to the equation $\L u = f $. For a
given positive $T$ we denote by $S_{T}$ the strip 
of $\R^{N+1}$ defined as follows
\begin{equation*}
	S_{T} = \R^{N} \times ]0,T[,
\end{equation*}
and we look for a classical solution to the Cauchy problem  
\begin{equation}
	\label{PdC}
	\begin{cases}
		\L u = f \qquad &{\rm in} \, S_{T}, \\
		u(\cdot, 0) = \phi \qquad &{\rm in} \, \R^{N}, 
	\end{cases}
\end{equation}
with $f \in C(S_{T})$ and $\phi \in C(\R^{N})$. Once again in view of \eqref{cauchyproblem} it is clear that 
growth condition on $f$ and $\phi$ are required to ensure existence and uniqueness for the solution to \eqref{PdC}. 
The following result is due to Di Francesco and Pascucci in \cite{PDF}.
\begin{theorem}
Let $\L$ be an operator of the form \eqref{mmL} under the assumptions \textbf{(H1)}, \textbf{(H2)}, \textbf{(H3)}. 
Consider the Cauchy problem \eqref{PdC} with $\phi \in C(\R^{N})$ and $f \in C^{\a}(\O)$, in the sense of Definition \ref{holdercontinuous}. Let us suppose for some positive constant $C$ 
\begin{equation*}
	\abs{f(x,t)} \le C \, e^{C \abs{x}^{2}} \qquad \abs{\phi(x)} \le C \; e^{C \abs{x}^{2}}.
\end{equation*}
for every $x \in \R^{N}$ and $0 < t < T$. Then there exists $0 < T_{0} \le T$ such that the function
\begin{equation}
	\label{u}
	u(x,t) = \int \limits_{\R^{N}} \G(x,t, \x, 0) \, \varphi (\x) \, d\x - 
	\int \limits_{0}^{t} \int \limits_{\R^{N}}  \G(x,t, \x, \t) \, f(\x, \t) \, d\x \, d\t.
\end{equation}
is well defined for every $(x,t) \in \R^{N} \times ]0,T_{0}[$. Moreover, it is a solution to the Cauchy problem \eqref{PdC}
and  the initial condition is attained by continuity
\begin{equation*}
	\lim \limits_{(x,t) \rightarrow (x_{0}, 0)} u(x,t) = \phi(x_{0}), 
			\qquad \text{for every} \; x_{0} \in \R^{N}.
\end{equation*}
\end{theorem}	

Uniqueness results for the Cauchy problem \eqref{PdC} can be found in \cite{Punic}, \cite{PDF} and \cite{PODF}. 
Later on, Cinti and Polidoro proved in \cite{CintiPolidoro} the following result. 
\begin{theorem}	
Let $\L$ be an operator of the form \eqref{mmL} under the assumptions \textbf{(H1)}, \textbf{(H2)}, \textbf{(H3)}. 
If $u$ and $v$ are two solutions to the same Cauchy problem \eqref{PdC} satisfying the following estimate 	
\begin{equation}
	\label{estimate-u}
	\int \limits_{0}^{T} \int \limits_{\R^{N}}
	\left( \abs{u(x,t)} + \abs{v(x,t)} \right) \, 
	e^{-C \left(\abs{x}^{2}  + \tfrac{1}{t^{\b}} \right) }\, dx \, dt
	< + \infty 
\end{equation}
with $0 < \b < 1$, then $u \equiv v$. 
\end{theorem}

We eventually quote the main uniqueness result of \cite{PODF}, that doesn't require any growth assumptions on the solutions $u$ and $v$. 
\begin{theorem}	
Let $\L$ be an operator of the form \eqref{mmL} under the assumptions \textbf{(H1)}, \textbf{(H2)}, \textbf{(H3)}. 
If $u$ and $v$ are two non-negative solutions to the same Cauchy problem \eqref{PdC}, with $f = 0$ and $\varphi \ge 0$, then $u \equiv v$. 
\end{theorem}

\subsection{The Dirichlet problem} 

In the sequel $\Omega$ will denote a bounded domain of $\R^{N+1}$. For every $f \in C(\O)$ and $\varphi \in C(\p \O, \R)$, we consider the Dirichlet problem for the operator $\L$ with H\"older continuous coefficients
\begin{equation}
    \label{PdD}
    \begin{cases}
        \L u = f \hspace{6mm} &\text{in} \hspace{1mm} \O, \\
        u = \varphi &\text{on} \hspace{1mm} \p \O.
    \end{cases}
\end{equation}
This problem has been studied by Manfredini in \cite{M} in the framework of the Potential Theory. In accordance with the usual axiomatic approach, we denote by $H^{\O}_{\varphi}$ the Perron-Wiener-Brelot-Bauer solution to the Dirichlet problem \eqref{PdD} with $f=0$. 
In order to discuss the boundary condition of the problem \eqref{PdD} we say that a point $z_{0} \in \p \O$ is $\L-$regular for $\O$ if 
\begin{equation} \label{boundaryreg}
	\lim \limits_{z \rightarrow z_{0}} H^{\O}_{\varphi}(z) \qquad {\rm for \, every} \, \varphi \in C(\p \O).
\end{equation}

The first result for the existence of a solution to the Dirichlet problem \eqref{PdD} for an operator $\L$ with H\"older continuous coefficiens is proved by Manfredini in \cite{M}, Theorem 1.4. 
\begin{theorem}
    \label{existence1}
Let $\L$ be an operator in the form \eqref{mmL} satisfying conditions \textbf{(H1), (H2), (H3)}, and assume that the matrix $B$ has the form \eqref{B0}. Suppose that $f \in C^\a (\overline \O)$ and $\varphi \in C(\p \O)$. Then there exixts a unique solution $u \in C^{2,\a}_\loc (\O)$ to the Dirichlet problem \eqref{PdD}. The function $u$ is a classical solution to $\L u = f$ in $\O$, and  $\lim \limits_{z \rightarrow z_0} u(z) = \varphi (z_0)$ for every $\L-$regular point $z_0 \in \p \O$.
\end{theorem}

The assumption that the matrix $B$ is of the form \eqref{B0} has been introduced to simplify the problem and seems to be unnecessary. Indeed, this condition is removed in \cite{PODF}, where a specific family of open sets $\Omega$ is considered. 
The uniqueness of the solution follows straightfarwardly from the following \emph{weak maximum principle} that can be found in the proof of Proposition 4.2 of \cite{M}.
\begin{theorem}
	\label{weak-max-princ}
Let $\L$ be an operator in the form \eqref{mmL} satisfying conditions \textbf{(H1), (H2), (H3)}, and assume that the matrix $B$ has the form \eqref{B0}. Let $\Omega$ be a bounded open set of $\R^{N+1}$, and let $u$ be a continuous function in $\overline{\O}$, such that  $\p_{x_{j}} u, \p_{x_{i}x_{j}}^{2} u$, for $i,j =1, \ldots, m_{0}$ and $Yu$ are continuous in $\O$. If moreover 
		\begin{equation*}
			\begin{cases}
                \L u\ge 0 & \text{in} \; \O, \\
				u \le 0 & \text{on} \; \p \O,
			\end{cases}
		\end{equation*}
	then $u \le 0$ in $\O$. 
\end{theorem}

\medskip

In order to discuss the boundary regularity of $\Omega$, we recall that the analogous of the Bouligand theorem for operators $\L$ has been proved in \cite{M}. Specifically, a point $z_{0} \in \p \O$ is $\L-$regular if there exists a local barrier at $z_{0}$, that is there exists a neighborhood $V$ of $z_{0}$ and a function $w \in C^{2,\a}(V)$ such that 
	\begin{equation*}
		\label{local_bar}
		w(z_{0}) = 0, \; w(z) > 0 \; \text{for } \, z \in \overline{\O \cap V} \setminus \{ z_{0} \} \qquad \text{and}
		\qquad \L w \le 0 \;  \text{in} \; \O \cap V.
	\end{equation*}

Let $z_0$ be point belonging to $\p \O$. We say that a vector $\nu \in \R^{N+1}$ is an \emph{outer normal to $\Omega$ at $z_0$} if there exists a positive $r$ such that $B(z_1, r |\nu|) \cap \overline{\Omega} = \{ z_0 \}$. Here $B(z_1, r |\nu|)$ is the Euclidean ball centered at $z_1 = z_0 + r \nu$ and radius $r |\nu|$. Note that this definition doesn't require any regularity on $\partial \O$ and several linearly independent vectors are allowed to be outer normal to $\Omega$ at the same point $z_0$. The following result proved in \cite{M} gives a very simple geometric condition for the boundary regularity of $\Omega$ and is in accordance with the Fichera's classification of $\p \O$. 
\begin{theorem}
    \label{Fichera}
Let $\L$ be an operator in the form \eqref{mmL} satisfying conditions \textbf{(H1), (H2), (H3)}. Consider the Dirichlet problem \eqref{PdD}, and let $z_0 \in \partial \Omega$. Assume that $\nu$ is an outer normal to $\Omega$ at $z_0$. Then it holds
\begin{itemize}
    \item if $\langle A(z_0) \nu, \nu \rangle \ne 0$, then there exists a local barrier at $z_{0}$;
    \item if $\langle A(z_0) \nu, \nu \rangle = 0$, and $\langle Y (z_0), \nu \rangle > 0$ then there exists a local barrier at $z_{0}$;
    \item if $\langle A(z_0) \nu, \nu \rangle = 0$, and $\langle Y (z_0), \nu \rangle < 0$ then $z_{0}$ is non regular.
\end{itemize}
\end{theorem}

\begin{center}
	\scalebox{0.9}{
\begin{pspicture*}(-5,-3.2)(5,1.5)
\pspolygon[fillstyle=solid,fillcolor=grey](-3.2,-1.95)(-1.2,-0.95)(3.2,-1.95)(1.2,-2.95)%

\pspolygon[fillstyle=solid,fillcolor=lgrey](-3.2,0.05)(-3.2,-1.95)(-1.2,-0.95)(-1.2,1)%
\pspolygon[fillstyle=solid,fillcolor=llgrey](-1.2,-0.95)(-1.2,1)(1.07,0.52)(1.07,-1.47)%

\pcline[linecolor=black](0,-2)(0,1.4)
\pcline[linecolor=black](0,1.4)(-.05,1.2)%
\pcline[linecolor=black](0,1.4)(.05,1.2)%
\uput[0](0,1.4){$t$}

\pspolygon[fillstyle=solid,fillcolor=lgrey](1.2,-2.95)(1.2,-0.95)(-0.9,-0.47)(-0.9,-2.47)%
\pspolygon[fillstyle=solid,fillcolor=grey](1.2,-2.95)(3.2,-1.95)(3.2,0.05)(1.2,-0.95)%

\pcline[linecolor=dgrey, linestyle=dashed](3.2,-1.95)(-1.2,-0.95)%

\pcline[linecolor=black, linestyle=dashed](0,-1.5)(.23,-0.81)%
\pcline[linecolor=black](.25,-0.75)(.5,0)%
\pcline[linecolor=black](.5,0)(.4,-.15)%
\pcline[linecolor=black](.5,0)(.5,-.2)%
\dotnode(0,-1.5){OOO}
\uput[0](0.2,-1.1){$Y$}

\pcline[linecolor=black](2,-1.2)(4,-1.7)%
\pcline[linecolor=black](4,-1.7)(3.8,-1.7)%
\pcline[linecolor=black](4,-1.7)(3.9,-1.6)%
\uput[0](4,-1.7){$\nu$}
\dotnode(2,-1.2){OO}

\pcline[linecolor=black](-4,1)(4,-1)%
\pcline[linecolor=black](4,-1)(3.8,-1)%
\pcline[linecolor=black](4,-1)(3.9,-.9)%
\uput[0](4,-1){$x_1$}
\pcline[linecolor=black](2,1)(-2,-1)%
\pcline[linecolor=black](-2,-1)(-1.8,-0.95)%
\pcline[linecolor=black](-2,-1)(-1.9,-0.9)%
\uput[180](-2,-1){$x_2$}
\dotnode(0,0){O}\nput{105}{O}{$(0,0,0)$}

\pcline[linecolor=black](-3.2,.05)(1.2,-.95)%
\pcline[linecolor=black](-1.2,1)(3.2,.05)%

\pcline[linecolor=black](0,-1.5)(-2, -2.5)%
\pcline[linecolor=black](-2,-2.5)(-1.8,-2.45)%
\pcline[linecolor=black](-2,-2.5)(-1.9,-2.4)%
\uput[0](-1.9, -2.65){$\nu$}

\end{pspicture*}
}
\end{center}
\begin{center}
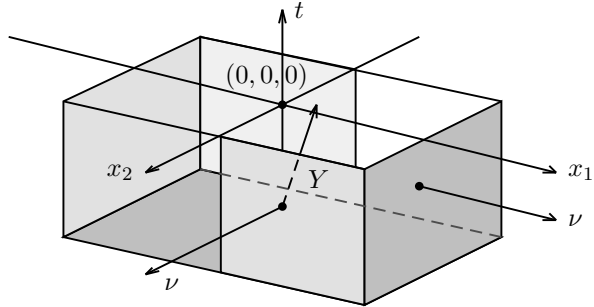
  {\scriptsize \sc Fig. 1 - Regular points for $\partial_{x_1}^2 + x_1 \partial_{x_2} - \partial_t$ on the set $]-1, 1[^{2} \times ]-1, 0[ $.}
\end{center}

The following more refined condition extends the Zaremba cone criterium. Let $\bar U$ be an open set of $\R^N$ and let $\bar t >0$. We denote by $Z_{\bar U, \bar t}(z_0)$ the following \emph{tusk-shaped} cone
\begin{equation*}
 Z_{\bar U, \bar t}(z_0) := \left\{ z_0 \circ D_r(\bar x,- \bar t) \mid \bar x \in \bar U, 0 \le r \le 1\right\}.
\end{equation*}
\begin{theorem}
    \label{Zaremba}
Let $\L$ be an operator in the form \eqref{mmL} satisfying conditions \textbf{(H1), (H2), (H3)}, and assume that the matrix $B$ has the form \eqref{B0}. Consider the Dirichlet problem \eqref{PdD}, and let $z_0 \in \partial \Omega$. If there exist $\bar U$ and $\bar t$ such that $Z_{\bar U, \bar t}(z_0) \cap \overline \O = \{ z_0 \}$, then there exists a local barrier at $z_{0}$.
\end{theorem}

Theorems \ref{Fichera} and \ref{Zaremba} have been first proved in \cite{M} assuming that the matrix $B$ has the form \eqref{B0}, this assumption has been removed from Theorem \ref{Fichera} in \cite{PODF}. We also recall the work \cite{LascialfariMorbidelli} by Lascialfari and Morbidelli, where a quasilinear problem is considered, and the article \cite{kogoj} by Kogoj for a complete treatment of the potential theory in the study of the Dirichlet problem for a general class of evolution hypoelliptic equations.

Recently, Kogoj, Lanconelli and Tralli prove in \cite{kogojlanconellitralli} a characterization of the $\L-$regular boundary points for constant coefficients operators $\L$ of the form \eqref{cost-op}. Their main result is stated in terms of a series involving $\L-$potentials of regions contained in $\R^{N+1} \setminus \O$, within different level sets of $\G$, the fundamental solution of $\L$. Specifically, if $F$ is a compact subset of $\R^{N+1}$, then $V_{F}$ denotes the $\L-$equilibrium potential of $F$. That is, 
\begin{equation}
		V_{F}(z) = \lim \inf \limits_{\z \rightarrow z} W_{F}(\z), \qquad z \in \R^{N+1},
\end{equation}
where if $\overline{\L}(\R^{N+1})$ denotes the family of $\L-$super harmonic functions in $\R^{N+1}$
\begin{equation}
		W_{F} := \inf \left\{ v: \, v \in \overline{\L}(\R^{N+1}), v \ge 0 \, {\rm in} \, \R^{N+1},
		v \ge 1 \, {\rm in} \, F \right\}.
\end{equation}
Moreover, for given $\mu \in ]0,1[, z_{0} \in \partial \O$, and for every positive integer $k$ we denote by $\O_{k}^{c}(z_{0})$ the set
\begin{equation*} 
	\O_{k}^{c}(z_{0}) := \Big\{ z \in \R^{N+1} \setminus \Omega \; \mid \; \big(\tfrac{1}{\mu} \big)^{k \log k} \le \G(z_{0}; z) \le \big(\tfrac{1}{\mu} \big)^{(k+1) \log (k+1)} \Big\}.
\end{equation*}
We then have (Theorem 1.1, \cite{kogojlanconellitralli}).
\begin{theorem}
Let $\L$ be an hypoelliptic operator in the form \eqref{cost-op}, let $\O$ be a bounded open subset of $\R^{N+1}$ and let $z_{0} \in \p \O$. Then $z_{0}$ is $\L-$regular for $\p \O$ if and only if
	\begin{equation}
		\sum \limits_{k=1}^{+\infty} V_{\O^{c}_{k}(z_{0})}(z_{0}) =+ \infty.
	\end{equation}
\end{theorem}	
We remark that this criterion is sharper than the Zaremba cone condition, moreover it provides us with a \emph{necessary} regularity condition. On the other hand, it only applies to constant coefficients operators in the form \eqref{cost-op}. 

\medskip

\section{Mean value formulas, Harnack inequalities and Strong Maximum Principle}
\label{sec4}
In the first part of this section we consider \emph{divergence form operators} acting on functions $u = u(x,t) \in C^{2,\alpha}(\Omega)$ as follows
\begin{equation}
    \label{divform}
    \L u = \sum \limits_{i,j=1}^{m_0} \p_{x_i} \left(a_{ij}(x,t) \p_{x_j} u \right) \, + \, 
        \sum \limits_{j=1}^{m_0} b_{j}(x,t) \p_{x_j} u \, + \,
        \langle B x, D u \rangle - \, \p_t u,
\end{equation}
under the structural assumptions \textbf{(H1), (H2), (H3)}. Moreover, we suppose the following additional assumption for the first order derivatives holds true:
\begin{itemize}
	\item[\textbf{(H4)}] for every $i,j = 1, \ldots, m_{0}$ the derivatives $\p_{x_{i}} a_{ij}(x,t)$ $\p_{x_{j}} b_{j}(x,t)$ exist and are bounded H\"older continuous functions of the exponent $\a$ in \textbf{(H3)}. 
\end{itemize}
The reason to consider classical solutions to divergence form operators is that the adjoint $\L^*$ of $\L$ is well defined 
and the function $\G^*(x,t,\xi,\t) = \G(\xi,\t, x,t)$ build via the parametrix method is the fundamental solution of $\L^*$. 


\subsection{Mean value formula}
The mean value formula we present here is based on the Green's identity and on the fundamental solution to $\L$ and is derived in the same way as for the classic parabolic case. In order to give the precise statement we need to introduce some notation. 
For every $r>0$ and for every $z_{0} \in \R^{N+1}$, we denote by $\O_{r}(z_{0})$ the \textit{super-level set of the fundamental solution} $\G$ of $\L$ defined as
\begin{equation}
	\label{super-level}
	\O_{r}(z_{0}) := \left\{ z \in \R^{N+1} \; \mid \; \G(z_{0}; z) > \tfrac1r \right\}.
\end{equation}
We remark that $\G$ is constructed via the parametrix method as the sum of a series of functions (see \eqref{parametrix} and \eqref{parametrix-G}), then the definition of the set $\O_{r}(z_{0})$ is implicit. However the parametrix method provides us with the following local estimate, useful to identify $\O_{r}(z_{0})$. For every $\e > 0$ there exists a positive $K$ such that
 \begin{equation}
		\label{stime}
		(1 - \e) \, Z (z_{0}, \z) \, \le \G(z_{0}, \z) \, \le \, (1 + \e ) \, Z (z_{0}, \z)
\end{equation}
for every $\z \in \R^{N+1}$ with $Z(z_{0}, \z) \ge K$, where $Z$ is the fundamental solution associated to the operator
$\L_{\z}$ defined in \eqref{ellez} and its explicit expression is available. Moreover, every super-level set of $Z$ is bounded whenever $B$ has the form \ref{B0}. This fact and Theorem \ref{teo31} imply that $\O_{r}(z_{0})$ is bounded for every sufficiently small positive $r$.

Mean value forulas for constant coefficients operators in the form \eqref{cost-op} have been proved by Kuptsov \cite{K3}, 
Garofalo and Lanconelli \cite{GL}, then by Lanconelli and Polidoro \cite{LP}. Later on, Polidoro considers operators $\L$ with H\"older continuous coefficients. Thus, we recall here Proposition 5.1 of \cite{P2}.


\begin{theorem}
	\label{mean-value-formula} 
Let $\L$ be an operator in the form \eqref{mmL} satisfying conditions \textbf{(H1), (H2), (H3)}, and assume that the matrix $B$ has the form \eqref{B0}. Let $u$ be a solution to $\L u = 0$ on $\O$. Then, for every $z_{0} \in \O$ such that $\overline{\O_{r}(z_{0})} \subset \O$, we have
	\begin{equation*}
		u(z_{0}) = \frac1r \; \int \limits_{\O_{r}(z_{0})} M(z_{0}; z ) \, u(z) \, dz. 
	\end{equation*}
Here
\begin{equation}
	\label{mean}
	M(z_{0}; z) = \frac{\langle A(z) \, D_{x} \G (z_{0}; z) \, , \, D_{x} \G (z_{0}; z) \rangle }{\G^{2}(z_{0}; z)}.
\end{equation}
\end{theorem}
As in Theorem \ref{existence1}, the assumption that the matrix $B$ has the form \eqref{B0} has been introduced to simplify the problem and seems to be unnecessary. We finally remark that mean value formulas analogous to the one stated in Theorem \ref{mean-value-formula}, where the kernel \eqref{mean} is replaced by a bounded continuous one, have been proved in \cite{K3}, \cite{GL}, \cite{LP} and \cite{P2}. Lastly, we recall a recent paper by Cupini and Lanconelli \cite{CupiniLanconelli}, where the authors give a general proof of Mean Value formulas for solutions to linear second order PDEs, only based on the local properties of the fundamental solution.

\subsection{Harnack inequality}
The first proofs of Harnack type inequalities for Kolmogorov operators have been derived using mean value formulas, and are due to Kuptsov \cite{K3} \cite{KU2}. This result has been improved by Garofalo and Lanconelli (see Theorem 1.1 in \cite{GL}) for some specific constant coefficients operators of the type \eqref{cost-op}. Their approach follows the ideas introduced for the heat equation by Pini \cite{Pini} and Hadamard \cite{HAD} in their seminal works.  Later on, Lanconelli and Polidoro proved the Harnack inequality for every operator \eqref{cost-op} considered in Section 2. The statement of this result requires a further notation. For every positive $\e$ we denote
\begin{equation}
	\label{super-level-K}
	K_{r}(z_{0}, \e) := \O_{r}(z_{0}) \cap \left\{ (x,t) \in \R^{N+1} \mid t \le t_0 - \e r^{2/Q} \right\}.
\end{equation}
We recall here Theorem 5.1 in \cite{LP}. 

\begin{theorem}
	\label{teo13}
Let $\L$ be an operator of the form \eqref{cost-op} satisfying the equivalent conditions of Proposition \ref{equivalence}. Then there exist three positive constants $c, r_{0} >0$ and $\e$, only dependent on $\L$, such that
\begin{equation}
 	\label{harnack-levelsets}
	\sup \limits_{z \in K_{r}(z_{0}, \e)} u(z) \, \le c u(z_{0}),
\end{equation}
for every non negative solution $u$ to $\L u=0$ in an open subset $\O$ of $\R^{N+1}$, for every $z_{0} \in \O$ such that $\overline{\O_{2r}(z_{0})} \subset \O$ and for every $r \in ]0,r_{0}[$.
\end{theorem}

The same result has been proved in \cite{P2} for variable coefficients operators \eqref{divform} satisfying \textbf{(H1)-
(H4)}, with $B$ in the form \eqref{B0}. We point out that the geometry of the above Harnack inequality is quite 
complicated. The natural analogy between the parabolic case and the Kolmogorov case is restored in \cite{LP}, where 
the Harnack inequality is written in terms of \emph{cylinders} (see equation \eqref{harnack-LP} below). Here and in the 
following, we consider the \textit{unit box} $\H$ defined as
\begin{equation}
	\label{unit-box}
	 \H = ]-1,1[^{N} \times ]-1,0[.
\end{equation}
Moreover, for given constants $\alpha, \beta, \gamma, \delta$ with $0 < \alpha < \beta < \gamma < 1$ and $0 < \d < 1$, 
we set
\begin{equation}
	\label{unit-box+-}
	 \H^+ = D_0(\delta)\left(]-1,1[^{N}\right) \times ]-\alpha,0[, \qquad 
	 \H^- = D_0(\delta)\left(]-1,1[^{N}\right) \times ]-\gamma,-\beta[.
\end{equation}

\begin{center}
	\scalebox{0.9}{
 		\begin{pspicture*}(-6,-3)(8,1.5)
			\pcline[linecolor=dgrey](-3,0)(1,-1)%
			\pcline[linecolor=dgrey](-3,-2)(1,-3)%
			\pcline[linecolor=dgrey](3,0)(-1,1)%
			\pcline[linecolor=dgrey, linestyle=dashed](3,-2)(-1,-1)%
			\pcline[linecolor=dgrey](3,0)(3,-2)%
			\pcline[linecolor=dgrey](-3,0)(-3,-2)%
			\pcline[linecolor=dgrey](1,-1)(1,-3)%
			\pcline[linecolor=dgrey, linestyle=dashed](-1,1)(-1,-1)%
			\pcline[linecolor=dgrey](3,0)(1,-1)%
			\pcline[linecolor=dgrey](-3,0)(-1,1)%
			\pcline[linecolor=dgrey](3,-2)(1,-3)%
			\pcline[linecolor=dgrey, linestyle=dashed](-1,-1)(-3,-2)
			\pcline[linecolor=dgrey](0,-2)(0,1.4)
			\pcline[linecolor=dgrey](0,1.4)(-.05,1.2)%
			\pcline[linecolor=dgrey](0,1.4)(.05,1.2)%
			\uput[0](0,1.4){$t$}
			\pspolygon[fillstyle=solid,fillcolor=lgrey](0.8,-0.12)(0.6,-0.22)(-0.8,0.12)(-0.6,0.22)
			\pspolygon[fillstyle=solid,fillcolor=lgrey](-0.8,0.12)(-0.8,-0.43)(0.6,-0.77)(0.6,-0.22)
			\pspolygon[fillstyle=solid,fillcolor=lgrey](0.8,-0.12)(0.6,-0.22)(0.6,-0.77)(0.8,-0.67)
			\pspolygon[fillstyle=solid,fillcolor=lgrey](0.8,-1.12)(0.6,-1.22)(-0.8,-0.88)(-0.6,-0.78)
			\pspolygon[fillstyle=solid,fillcolor=lgrey](-0.8,-0.88)(-0.8,-1.43)(0.6,-1.77)(0.6,-1.22)
			\pspolygon[fillstyle=solid,fillcolor=lgrey](0.8,-1.12)(0.6,-1.22)(0.6,-1.77)(0.8,-1.67)
			\pcline[linecolor=dgrey,linestyle=dashed](0,-2)(0,1.4)
			\pcline[linecolor=dgrey](-4,1)(4,-1)%
			\pcline[linecolor=dgrey](4,-1)(3.8,-1)%
			\pcline[linecolor=dgrey](4,-1)(3.9,-.9)%
			\uput[0](4,-1){$x_1$}
			\pcline[linecolor=dgrey](2,1)(-2,-1)%
			\pcline[linecolor=dgrey](-2,-1)(-1.8,-1)%
			\pcline[linecolor=dgrey](-2,-1)(-1.9,-.9)%
			\uput[180](-2,-1){$x_2$}
			\uput[90](-1.8,0.8){$\H$}
			\uput[60](0.8,-0.25){$\H^+$}
			\uput[300](-1,-1.5){$\H^-$}
	    \end{pspicture*}
	}
\end{center}
\begin{center}
  {\scriptsize \sc Fig. 3 - Harnack inequality.}
\end{center}

Based on the translation and on the dilation respectively defined in \eqref{law} and \eqref{fam-dil}, we introduce for every 
$r>0$ the cylinders
\begin{align*}
	& \H_{r}  := D(r) \H =  \left\{ D(r)(x,t) \mid (x,t) \in \H \right\}  \\
	& \H_{r} (x_{0}, t_{0}) := (x_{0}, t_{0}) \circ \H_r \\
	& \qquad \qquad \ \, = \left\{ (x_{0}, t_{0}) \circ D(r)(x,t) \mid (x,t) \in \H \right\} 
\end{align*}
centered at the origin and at a point $(x_{0}, t_{0}) \in \R^{N+1}$, respectively. Analogously, we define 
\begin{equation*}
 \H_{r}^+ (x_{0}, t_{0}) := (x_{0}, t_{0}) \circ D(r) \H^+, \qquad 
\H_{r}^- (x_{0}, t_{0}) := (x_{0}, t_{0}) \circ D(r) \H^-.
\end{equation*}

Given the above notation, we recall that in Theorem 5.1' of \cite{LP} is proved a Harnack inequality analogous to 
\eqref{harnack-levelsets}, where the sets $\overline{\O_{2r}(z_{0})}$ and $K_{r}(z_{0}, \e)$ are replaced by cylinders. 
Specifically, we have 
    \begin{equation}
	 	\label{harnack-LP}
		\sup \limits_{z \in \H^{-}_{r}(z_0)} u(z) \, \le c \, u(z_0),
    \end{equation}
whenever $\overline{\H_{r}(z_{0})} \subset \O$. 
We next quote the most general Harnack inequality for operators in non-divergence form as defined in \eqref{mmL}
proved in \cite{PODF}.
\begin{theorem}
	\label{harnack-podf}
Let $\L$ be an operator of the form \eqref{mmL} satisfying \textbf{(H1)-(H3)}.Then there exist positive constants $c, r_{0}, \a, \b, \g$ and $\d$, only dependent on the parameters of the assumptions \textbf{(H1)-(H3)}, such that
	 \begin{equation}
	 	\label{harnack-difra}
		\sup \limits_{z \in \H^{-}_{r}(z_0)} u(z) \, \le c \inf \limits_{z \in \H^{+}_{r}(z_0)} u(z),
	 \end{equation}
	 for every non negative solution $u$ to $\L u=0$ in an open subset $\O$ of $\R^{N+1}$, for every $z_{0} \in \O$ 
	 such that $\overline{\H_{r}(z_{0})} \subset \O$ and for every $r \in ]0,r_{0}[$.
\end{theorem}

\medskip

In spite of their local nature, Harnack inequalities are essential tools for the proof of non-local results. Among them, we find the Liouville theorems proved by Kogoj and Lanconelli in \cite{KogojLanconelli1, KogojLanconelli2} and the ones proved by Kogoj, Pinchover and Polidoro in \cite{KogojPinchoverPolidoro}. Moreover, they are also used to derive asymptotic estimates for positive solutions by a repeated application of them. Harnack chains are the tool needed to prove this kind of estimates. 

\medskip

\noindent 
{\sc Harnack chain.} \textit{We say that a finite sequence $(x_0,t_0), (x_1,t_1), \dots, (x_k,t_k)$ is a \emph{Harnack chain} if there exist  positive constants $r_0, r_1, \dots, r_{k-1}$ such that $\H_{r_j}(x_j,t_j) \subset \Omega$ and $(x_{j+1},t_{j+1}) \in  \H_{\theta r_j}(x_j,t_j)$ for $j=0, \dots, k-1$, so that, by the repeated use of the Harnack inequality, we obtain 
\begin{equation*}
 u(x_k,t_k) \le c u (x_{k-1},t_{k-1}) \le \dots \le c^k u(x_0,t_0),
\end{equation*}
for every non-negative solution $u$ to $\L u = 0$ in $\Omega$.
}

\medskip

In particular, a first application of this tool can be found in the proof of Proposition \ref{proposition} in the following 
subsection, where Harnack chains are used to prove a geometric version of Theorem \ref{harnack-podf}. Further 
applications can be found in the papers by Polidoro \cite{P}, Di Francesco and Polidoro \cite{PODF}, Boscain and 
Polidoro \cite {BoscainPolidoro} and Cibelli and Polidoro \cite{CibelliPolidoro} to obtain asymptotic estimates for the 
fundamental solution. We also recall the work by Cinti, Nystr\"om and Polidoro \cite{CNP12, CNP13} where a 
boundary Harnack inequality is proved. 

\subsection{Strong Maximum Principle}

The most general statement of the strong maximum principle for \emph{subsoltions} to Kolmogorov equations is proved by Amano in \cite{Amano79}. It extends the Bony's \emph{maximum propagation principle} \cite{Bony69} to a wide family of possibly degenerate operators with coefficients $a_{ij} \in C^{1}$, among which we find the ones in the form \eqref{mmL}. To our knowledge, a proof of the strong maximum principle for operators of the form \eqref{mmL} with continuous coefficients $a_{ij}$'s is not available in literature, even though it is expected to be true. For this reason, in the following we derive from Theorem \ref{harnack-podf}
a strong maximum principle for \emph{solutions} to $\L u = 0$, assuming that the coefficients $a_{ij}$'s are H\"older continuous.

In order to state the strong maximum principle, we introduce the notion of $\L$-\textit{admissible curve} and that of $\L$-\textit{admissible set}. 
Recall that to every operator $\L$ in the form \eqref{mmL} we associate the model operator \eqref{mmLap}, 
which can be written in the H\"ormander form
\begin{equation*}
 \sum_{j=1}^{m_0} X_j^2 + Y, \quad \text{with} \quad X_j = \p_{x_j} \quad \text{for} \quad j=1, \dots, m_0.
\end{equation*}

\begin{definition}
	\label{admissible-curve}
Let $\L$ be an operator of the form \eqref{mmL}, satisfying assumptions \textbf{(H1)-(H3)}. We say that a curve $ \g: [0,T] \rightarrow \R^{N+1}$ is \emph{$\L$-admissible} if is absolutely continuous and 
\begin{equation*}
	 \dot{\g}(s) = \sum\limits_{k=1}^{m_0} \omega_k(s) X_k(\g(s)) + Y(\gamma(s))
\end{equation*}	
for almost every $s \in [0,T]$ and with $\o_1, \o_2, \dots, \o_{m_0} \in L^{1}[0,T]$.
\end{definition}

\medskip

\begin{definition} \label{prop-set}
Let $\O$ be any open subset of $\R^{N+1}$, and let $\L$ be an operator of the form \eqref{mmL}, satisfying assumptions \textbf{(H1)-(H3)}. For every point $(x_0,t_0) \in \O$ we denote by $\AS ( \O )$ the \emph{attainable set} defined as
\begin{equation*}
	\AS  ( \O ) = 
	\begin{Bmatrix}
	(x,t) \in \O \mid \hspace{1mm} \text{\rm there exists an} \ \L - \text{\rm admissible curve}
	\ \g : [0,T] \rightarrow \O \hspace{1mm} \\ 
	\hfill \text{\rm such that} \ \g(0) = (x_{0}, t_{0}) \hspace{1mm} {\rm and}  
	\hspace{1mm} \g(T) = (x,t)
	\end{Bmatrix}.
\end{equation*}
Whenever there is no ambiguity on the choice of the set $\O$ we denote $\AS = \AS ( \O )$. 
\end{definition}

\medskip

We are now in position to state the strong maximum principle. 
\begin{theorem} \label{maincor}
	Let $\O$ be any open subset of $\R^{N+1}$, and let $\L$ be an operator of the form \eqref{mmL}, satisfying assumptions \textbf{(H1)-(H3)}. 
	Let $u \ge 0$ be a solution to $\L u = 0$ in $\O$. If $u(x_{0}, t_{0})=0$ for some point $(x_{0}, t_{0}) \in \O$ , then $u(x,t) = 0$ for every $(x,t) \in \overline{\AS}$.
\end{theorem}

We remark that the attainable set $\AS$ strongly depends on the domain $\O$. For instance, when $\O$ agrees with the unit box $\H = ]-1,1[^{2} \times ]-1,0[$ we have 
\begin{equation} \label{eq-prop-set} 
	\A_{(0,0,0)} = \Big\{ (x_1,x_2,t) \in \H \mid |x_1| \le |t| \Big\}.
\end{equation} 
\begin{center}
	\scalebox{0.9}{
 		\begin{pspicture*}(-6,-3)(8,2)
			\pcline[linecolor=dgrey](0,-2)(0,1.4)
			\pcline[linecolor=dgrey](0,1.4)(-.05,1.2)%
			\pcline[linecolor=dgrey](0,1.4)(.05,1.2)%
			\uput[0](0,1.4){$t$}
			\pcline[linecolor=dgrey](-4,1)(4,-1)%
			\pcline[linecolor=dgrey](4,-1)(3.8,-1)%
			\pcline[linecolor=dgrey](4,-1)(3.9,-.9)%
			\uput[0](4,-1){$x_{2}$}
			\pcline[linecolor=dgrey](2,1)(-2,-1)%
			\pcline[linecolor=dgrey](-2,-1)(-1.8,-1)%
			\pcline[linecolor=dgrey](-2,-1)(-1.9,-.9)%
			\uput[180](-2,-1){$x_{2}$}
			\pcline[linecolor=dgrey](-3,0)(1,-1)%
			\pcline[linecolor=dgrey](-3,-2)(1,-3)%
			\pcline[linecolor=dgrey](3,0)(-1,1)%
			\pcline[linecolor=dgrey, linestyle=dashed](3,-2)(-1,-1)%
			\pcline[linecolor=dgrey](3,0)(3,-2)%
			\pcline[linecolor=dgrey](-3,0)(-3,-2)%
			\pcline[linecolor=dgrey](1,-1)(1,-3)%
			\pcline[linecolor=dgrey, linestyle=dashed](-1,1)(-1,-1)%
			\pcline[linecolor=dgrey](3,0)(1,-1)%
			\pcline[linecolor=dgrey](-3,0)(-1,1)%
			\pcline[linecolor=dgrey](3,-2)(1,-3)%
			\pcline[linecolor=dgrey, linestyle=dashed](-3,-2)(-1,-1)%
			\pcline[linecolor=dgrey](-2,.5)(2,-.5)%
			\pcline[linecolor=dgrey, linestyle=dashed](0,0)(0,-2)%
			\pcline[linecolor=dgrey](2,-.5)(1,-3)%
			\pcline[linecolor=dgrey, linestyle=dashed](-2,.5)(-1,-1)%
			\pspolygon[fillstyle=solid,fillcolor=lgrey](-2,.5)(2,-.5)(1,-3)(-3,-2)
			\pspolygon[fillstyle=solid,fillcolor=grey](2,-.5)(3,-2)(1,-3)
			\pcline[linecolor=dgrey](2,1)(-2,-1)%
			\pcline[linecolor=dgrey](-2,-1)(-1.8,-1)%
			\pcline[linecolor=dgrey](-2,-1)(-1.9,-.9)%
			\uput[180](-2,-1){$x_{1}$}
			\uput[90](-1.8,0.8){$\H$}
			\pcline[linecolor=dgrey](-3,0)(1,-1)%
			\pcline[linecolor=dgrey](-3,-2)(1,-3)%
			\pcline[linecolor=dgrey](3,0)(-1,1)%
			\pcline[linecolor=dgrey, linestyle=dashed](3,-2)(-1,-1)%
			\pcline[linecolor=dgrey](3,0)(3,-2)%
			\pcline[linecolor=dgrey](-3,0)(-3,-2)%
			\pcline[linecolor=dgrey](1,-1)(1,-3)%
			\pcline[linecolor=dgrey, linestyle=dashed](-1,1)(-1,-1)%
			\pcline[linecolor=dgrey](3,0)(1,-1)%
			\pcline[linecolor=dgrey](-3,0)(-1,1)%
			\pcline[linecolor=dgrey](3,-2)(1,-3)%
			\pcline[linecolor=dgrey, linestyle=dashed](-3,-2)(-1,-1)%
			\pcline[linecolor=dgrey](-2,.5)(2,-.5)%
			\pcline[linecolor=dgrey, linestyle=dashed](0,0)(0,-2)%
			\pcline[linecolor=dgrey](2,-.5)(1,-3)%
			\pcline[linecolor=dgrey, linestyle=dashed](-2,.5)(-1,-1)%
			\dotnode(0,0,0){O}
			\uput[115](0.16,0){$(0,0,0)$}
		\end{pspicture*}}
		\end{center}
		\begin{center}
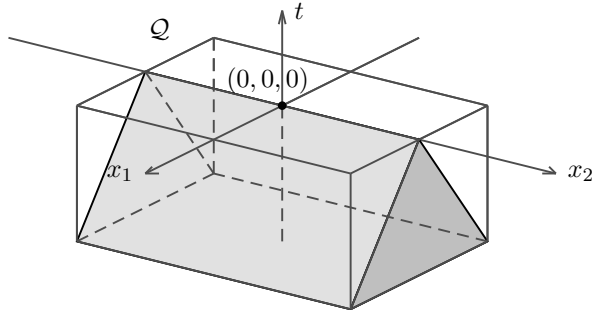
  {\scriptsize \sc Fig. 2 - $\mathscr{A}_{(0,0,0)}(\H)$ with $\H =]-1, 1[^{2} \times ]-1, 0[ $.}
\end{center}

For the proof of this fact we refer to \cite{CNP10}, Proposition 4.5, p.$353$. Moreover, the statement of Theorem \ref{maincor} is optimal. Indeed, in Proposition 4.5 of \cite{CNP10} it is also shown that there exists a non-negative solution $u$ to $\L u = 0$ in $\H$ such that $u(x,t) = 0$ for every  $(x,t) \in \overline{\A_{(0,0)}}$, and $u(x,t) > 0$ for every  $(x,t) \in \H \backslash \overline{\A_{(0,0)}}$.

\medskip 

In order to prove Theorem \ref{maincor}, we first need to prove the following intermediate result.
\begin{theorem} \label{mainthe}
Let $\L$ be an operator of the form \eqref{mmL} satisfying \textbf{(H1)-(H3)}, and let $\O$ be an open subset of 
$\R^{N+1}$. For every $z_{0} \in \O$, and for any compact set $K \subseteq {\rm int} \big(\AS \big)$, there exists a 
positive constant $C_{K}$, 
only dependent on $\O$, $z_{0}$, $K$ and on the operator $\L$, such that 
\begin{equation*}
  \sup_{z \in K} u(z) \, \le \, C_{K} \, u(z_{0}),
\end{equation*}
for every non negative solution $u$ to $\L u=0$ in $\O$.
\end{theorem}
We then obtain, as a direct consequence, the proof of the Strong Maximum Principle stated in Theorem \ref{maincor}. In order to achieve this program, we introduce a further notation and we recall a lemma, whose proof can be found in Lemma 2.2 of \cite{BoscainPolidoro}. 
Given $\b, \delta$ as in the definition of $\H^{-}$ and for every $z \in \R^{N+1}$, $r>0$ we set 
\begin{equation*}
\begin{split}
  & \widetilde \H := ]-1,1[^{N+1} \qquad 
  \widetilde \H_{r} (x_{0}, t_{0}) := (x_{0}, t_{0}) \circ D(r) \widetilde \H; \\
  & K^- = D_0(\delta)\left(]-1,1[^{N}\right) \times \big\{ -\tfrac{\b+\g}{2} \big\} \qquad 
  K_{r}^- (x_{0}, t_{0}) := (x_{0}, t_{0}) \circ D(r) K^-.
\end{split}
\end{equation*}

\begin{lemma} \label{lemma2.2}
Let $\g: [0,T] \rightarrow \R^{N+1}$ be an $\L-$admissible path and let $a, b$ be two constants s.t. $0 \le a < b \le T$. 
Then there exists a positive constant $h$, only depending on $\L$, such that
\begin{equation*}
		\int_{a}^{b} |\o (\t)|^{2} \d\t \le h \hspace{4mm} \implies 
		\hspace{4mm} \g(b) \in K^{-}_{r}(\g(a)), \hspace{2mm} {\rm with} \hspace{2mm}
		r = \sqrt{ 2 \frac{b - a}{\b + \g}}.
\end{equation*}
\end{lemma}

Note that $K^-_r(z)$ is a subset of $\H^-_r(z)$, then Lemma \ref{lemma2.2} implies that $\H^{-}_{r}(\g(a))$ is an open neighborhood of $\g(b)$.
Our first result of this section is a local version of Theorem \ref{harnack-podf}, whose proof only relies on the Harnack chains  and on Lemma \ref{lemma2.2}.

\begin{proposition} \label{proposition}
Let $z_{0}$ be a point of $\O$, an open subset of $\R^{N+1}$. For every $z \in {\rm int} \big( \A_{z_{0}} \big)$ there exist an open neighborhood $U_{z}$ of $z$ and a positive constant $C_{z}$ such that
\begin{equation*}
		\sup_{U_{z}} u \, \le \, c_{z} \, u(z_{0})
\end{equation*}
for every non-negative solution $u$ to $\L u=0$ in an open subset $\O$ of $\R^{N+1}$.
\end{proposition}

\begin{proof}
Let $z$ be any point of $\text{int} \big( \A_{z_0} \big)$. We plan to prove our claim by 
constructing a finite Harnack chain connecting $z$ to $z_0$. Because of the very definition 
of $\A_{z_0}$, there exists a $\L-$admissible curve $\g : [0, T] \rightarrow \O$ steering $z_0$ to $z$. 
Our Harnack chain will be a finite subset of $\g([0,T])$. As $\widetilde \H_{r} (x_{0}, t_{0})$ is an open neighborhood of $(x_{0}, t_{0})$, for every $s \in [0,T]$ we can set 
\begin{equation} \label{rs}
		r(s) \, := \, \sup \left\{ r > 0 \, : \, \widetilde \H_{r}(\g(s)) \subseteq \O \right\}.
\end{equation}
Note that the function \eqref{rs} is continuous, then it is well defined the positive number 
\begin{equation} \label{r}
	r_0 \hspace{1mm} := \hspace{1mm} \min_{s \in [0,T]} r(s). 
\end{equation}
Moreover $\H_{r}(\g(s)) \subset \widetilde \H_{r}(\g(s))$, then
\begin{equation} \label{Qr}
	\H_{r}(\g(s)) \subseteq \O \quad \text{for every} \ s \in [0,T] \quad \text{and} \ r \in ]0, r_0]. 
\end{equation}
On the other hand, we notice that the following function is (uniformly) continuous in $[0,T]$
\begin{equation}
	I(s) \hspace{1mm} := \hspace{1mm} \int_{0}^{s} | \o (\t)|^{2} dt, 
\end{equation}
then there exists a positive $\delta_0$ such that $\delta_0 \le \b  r_0$ and that 
\begin{equation} \label{delta}
	\int_{a}^{b} | \o (\t)|^{2} dt \le h \qquad \text{for every} \ 
	a,b \in [0,T], \quad \text{such that} \  0 < a-b \le \d_0, 
\end{equation}
where $h$ is the constant appearing in Lemma \ref{lemma2.2}. 

We are now ready to construct our Harnack chain. Let $k$ be the unique positive integer such that $(k-1) \d_0 < T$, 
and $k \d_0 \ge T$. We define $\{ s_{j} \}_{j \in \{ 0, 1, \ldots, k \}} \in [0,T]$ as follows: $s_j = j \d_0$ for $j=0,1, \dots, 
k-1$, and $s_k =T$.
As noticed before, the equation \eqref{delta} allows us to apply Lemma \ref{lemma2.2}. We then obtain
\begin{equation} \label{propsuccessione}
	\g(s_{j+1}) \in \H^{-}_{{r_0}} (\g(s_{j}))  \hspace{5mm} j=0, \ldots, k-2, \qquad \g(s_k) \in \H^{-}_{{r_1}} 
	(\g(s_{k-1})),
\end{equation}
for some $r_1 \in ]0, r_0]$. We next show that $\left( \g(s_{j}) \right)_{j=0,1, \dots,k}$ is a Harnack chain and we 
conclude the proof.  We proceed by induction. 
For every $j = 1, \ldots, k-2$ we have that $\g (s_{j+1}) \in \H^{-}_{{r_0}} (\g(s_{j}))$. From \eqref{Qr} we know that 
$\H_{{r_0}} (\g(s_{j})) \subseteq \Omega$, then we apply Theorem \ref{harnack-podf} and we find
 
\begin{align*}
	u(\g(s_{j+1})) \le \sup_{\H^{-}_{{r_0}}(\g (s_{j}))} u \le \, c 
	 \inf_{\H^{+}_{{r_0}}(\g (s_{j}))} u 
	\le c  u(\g (s_{j}).
\end{align*}
As a consequence, we obtain
\begin{align*}
u(\g (s_{k-1})) &\le c u(\g(s_{k-2}))  \le M^{2} u(\g(s_{k-3}))  \le \ldots \le c^{k-1} u(\g(0)).
\end{align*}
We eventually apply Theorem \ref{harnack-podf} to the set $\H_{{r_1}} (\g(s_{k-1})) \subseteq \Omega$ and we obtain
\begin{equation*}
		\sup_{U_{z}} u \, \le \, c^{k} \,  u(z_{0}) ,
\end{equation*}
where $U_{z} = \H^{-}_{{r_1}} (\g(s_{k-1}))$. As we noticed above,
$\H^{-}_{{r_1}} (\g(s_{k-1}))$ is an open neighborhood of $\g(T)$. This concludes the proof.
\end{proof}

\medskip

\noindent {\sc Proof of Theorem \ref{mainthe}}.
Let  $K$ be any compact subset of ${\rm int} \left( \A_{z_0} \right)$.  For every $z \in K$ we 
consider the open set $U_{z}$ appearing in the statement of Proposition \ref{proposition}. Clearly we have
\begin{equation*}
	K \, \subseteq \, \bigcup_{z \in K} \hspace{1mm} U_{z}.
\end{equation*}
Because of its compactness, there exists a finite covering of $K$
\begin{equation*}
	K \, \subseteq \, \bigcup_{j=1, \ldots, m_{K}} \hspace{1mm} U_{z_{j}},
\end{equation*}
and Proposition \ref{proposition} yields
\begin{equation*}
	\sup_{U_{z_{j}}} u \le C_{z_{j}} \,
	u(z_{0}) \qquad j=1, \ldots, m_{K}.
\end{equation*} 
This concludes the proof of Theorem \ref{mainthe}, if we choose
\begin{equation*}
	C_{K} \, = \, \max_{j=1, \ldots, m_{K}} C_{z_{j}}.
\end{equation*}
\hfill $\square$

\medskip

\noindent {\sc Proof of Theorem \ref{maincor}}. 
If $u$ is a non-negative solution to $\L u = 0$ in $\Omega$ and $K$ is a compact subset of ${\rm int} ( \A_{z_{0}} )$, then $\sup_{K} u \le C_{K} u(z_{0})$. If moreover $u(z_{0}) = 0$, we have $u(z) = 0$ for every $z \in K$ and, thus, for every $z \in \A_{z_{0}}$. The conclusion of the proof then follows from the continuity of $u$. \hfill $\square$

\section*{Acknowledgements}
The authors thank the anonymous referee for her/his suggestions that improved this manuscript. This research was partially supported by the grant of Gruppo Nazionale per l'Analisi Matematica, la Probabilità e le loro Applicazioni (GNAMPA) of the Istituto Nazionale di Alta Matematica (INdAM). The second author acknowledges financial support from the FAR2017 project ``The role of Asymmetry and Kolmogorov equations in financial Risk Modelling (ARM)”.

\end{document}